\documentclass[11pt]{article}
\usepackage{mathrsfs}
\newenvironment{proof}[1][Proof]{\begin{trivlist}
\item[\hskip \labelsep {\bfseries #1}]}{\end{trivlist}}
\newcounter{example}

\newtheorem{thm}{Theorem}[section]

\newtheorem{prop}{Proposition}
\usepackage{latexsym}
\usepackage{amsmath}
\usepackage{amssymb}
\usepackage{graphicx}
\usepackage{textcomp}
\usepackage{multirow}
\usepackage{enumerate}
\usepackage{cancel}
\usepackage{hyperref}
\usepackage{caption}  
\usepackage[usenames,dvipsnames]{color}
\usepackage[left, pagewise]{lineno}
\usepackage[ruled,vlined]{algorithm2e}
\usepackage{epstopdf}
\DeclareMathOperator{\sech}{sech}
\title{Structure-Preserving Galerkin POD Reduced-Order Modeling of Hamiltonian Systems}
\author{Yuezheng Gong\thanks{Applied and Computational Mathematics Division, Beijing Computational Science Research Center, Beijing 100193, China. Email: {gongyuezheng@csrc.ac.cn}.}
\and Qi Wang \thanks{Department of Mathematics, University of South Carolina, Columbia, SC 29208.  Email: {qwang@math.sc.edu}. }
\and Zhu Wang \thanks{ Corresponding author.
           Department of Mathematics, University of South Carolina, Columbia, SC 29208.
           Email: {wangzhu@math.sc.edu}. }
}
\date{}
\begin{document}
\maketitle
\begin{abstract}
The proper orthogonal decomposition reduced-order models (POD-ROMs) have been widely used as a computationally efficient surrogate models in large-scale numerical simulations of complex systems.
However, when it is applied to a Hamiltonian system, a naive application of the POD method can destroy its Hamiltonian structure in the reduced-order model.
In this paper, we develop a new reduce-order modeling approach for the Hamiltonian system, which uses the traditional framework of Galerkin projection-based model reduction but modifies the ROM so that the appropriate Hamiltonian structure is preserved.
Since the POD truncation can degrade the approximation of the Hamiltonian function, we propose to use the POD basis from shifted snapshots to improve the Hamiltonian function approximation.
We further derive a rigorous {\em a priori} error estimate of the structure-preserving ROM and demonstrate its effectiveness in several numerical examples.
This approach can be readily extended to dissipative Hamiltonian systems, port-Hamiltonian systems etc.
\end{abstract}



\newcommand{\lp}{\left(}
\newcommand{\rp}{\right)}
\newcommand{\lno}{\left\|}
\newcommand{\rno}{\right\|}
\newcommand{\id}{\text{ d}}

\newcommand{\ou}{\overline{u}}
\newcommand{\opsi}{\overline{\psi}}
\newcommand{\oq}{\overline{q}}
\newcommand{\oT}{\overline{T}}
\newcommand{\E}{\mathbbm{E}}
\newcommand{\orho}{\overline{\rho}}

\newcommand{\s}{\sigma}
\renewcommand{\k}{\kappa}
\newcommand{\p}{\partial}
\newcommand{\om}{\omega}
\newcommand{\Om}{\Omega}
\newcommand{\pOm}{\partial \Omega}
\newcommand{\e}{\epsilon}
\renewcommand{\a}{\alpha}
\renewcommand{\b}{\beta}
   \newcommand{\eps}{\varepsilon}
   \newcommand{\EX}{{\Bbb{E}}}
   \newcommand{\PX}{{\Bbb{P}}}

\newcommand{\nd}{{\nabla \cdot}}

\newcommand{\cF}{{\cal F}}
 
\newcommand{\cD}{{\cal D}}
\newcommand{\cO}{{\cal O}}

\newtheorem{remark}{Remark}[section]
\newtheorem{lemma}{Lemma}[section]
\newtheorem{theorem}{Theorem}[section]
\newtheorem{corollary}{Corollary}[section]
\newtheorem{proposition}{Proposition}[section]
\newtheorem{definition}{Definition}[section]
\newcommand{\kmodel}{k_{\mbox{Model}}}
\newcommand{\obu}{\overline{\bf u}}
\newcommand{\oobu}{\overline{\overline{\bf u}}}
\newcommand{\be}{{\bf e}}
\newcommand{\bk}{{\bf k}}
\newcommand{\bs}{{\bf s}}
\newcommand{\bu}{{\bf u}}
\newcommand{\bD}{{\bf D}}
\newcommand{\bN}{{\bf N}}
\newcommand{\bS}{{\bf S}}
\newcommand{\oou}{\overline{\overline{u}}}
\newcommand{\op}{\overline{p}}
\newcommand{\of}{\overline{f}}
\newcommand{\obf}{\overline{\bf f}}
\newcommand{\ow}{\overline{w}}
\newcommand{\ov}{\overline{v}}
\newcommand{\ophi}{\overline{\phi}}
\newcommand{\oS}{\overline{S}}
\newcommand{\obS}{\overline{\bf S}}
\newcommand{\bv}{{\bf v}}
\newcommand{\obv}{\overline{\bf v}}
\newcommand{\bc}{{\bf c}}
\newcommand{\by}{{\bf y}}
\newcommand{\bA}{{\bf A}}
\newcommand{\bB}{{\bf B}}
\newcommand{\bI}{{\bf I}}
\newcommand{\bY}{{\bf Y}}
\newcommand{\bw}{{\bf w}}
\newcommand{\bW}{{\bf W}}
\newcommand{\bU}{{\bf U}}
\newcommand{\obw}{\overline{\bf w}}
\newcommand{\bz}{{\bf z}}
\newcommand{\bZ}{{\bf Z}}
\newcommand{\obZ}{\overline{\bf Z}}
\newcommand{\bff}{{\bf f}}
\newcommand{\bee}{{\bf e}}
\newcommand{\bn}{{\bf n}}
\newcommand{\bx}{{\bf x}}
\newcommand{\bX}{{\bf X}}
\newcommand{\bH}{{\bf H}}
\newcommand{\bV}{{\bf V}}
\newcommand{\bL}{{\bf L}}
\newcommand{\bg}{{\bf g}}
\newcommand{\bj}{{\bf j}}
\newcommand{\br}{{\bf r}}
\newcommand{\grads}{\nabla^s}
\def\PP{{{\rm l}\kern - .15em {\rm P} }}
\def\PN2{{\PP_{N}-\PP_{N-2}}}
\newcommand{\erf}[1]{\mbox{erf}\left(#1\right)}
\newcommand{\D}{\nabla}
\newcommand{\I}{\mathbb{I}}
\newcommand{\N}{\mathbb{N}}
\newcommand{\R}{\mathbb{R}}
\newcommand{\Z}{\mathbb{Z}}
\newcommand{\mathR}{\R}
\newcommand{\mathN}{\N}
\newcommand{\mathZ}{\Z}
\newcommand{\mathI}{\mathbbm{I}}
\newcommand{\btau}{\boldsymbol{\tau}}
\newcommand{\bphi}{\boldsymbol{\phi}}
\newcommand{\bvarphi}{\boldsymbol{\varphi}}
\newcommand{\bpsi}{\boldsymbol{\psi}}
\newcommand{\bfeta}{\boldsymbol{\eta}}
\newcommand{\blambda}{\boldsymbol{\lambda}}
\newcommand{\bPhi}{\boldsymbol{\Phi}}
\newcommand{\bPsi}{\boldsymbol{\Psi}}
\newcommand{\obphi}{\overline {\boldsymbol{\phi}}}
\newcommand{\bomega}{\boldsymbol{\omega}}
\newcommand{\bsigma}{\boldsymbol{\sigma}}
\newcommand{\orhoprime}{\overline{\rho^{\prime}}}
\newcommand{\bus}{{\bf u}^*}
\newcommand{\By}{\mathcal B(\by)}
\newcommand{\eci}[1]{\mathcal E_{#1}}
\newcommand{\dpyi}[1]{\delta_{#1}^+(\by)}
\newcommand{\dmyi}[1]{\delta_{#1}^-(\by)}
\newcommand{\cA}{{\mathcal A(\by)}}
\newcommand{\dyi}[1]{\delta_{#1}(\by)}
\newcommand{\cG}{{\mathcal G(\bx,\by)}}
\newcommand{\cGi}[1]{{\mathcal G_{#1}(\bx,\by)}}
\newcommand{\pti}{\partial_i}
\newcommand{\ptii}[1]{\partial_{#1}}
\newcommand{\ba}{{\bf a}}

\newcommand{\rey}{\mbox{Re}}

\newcommand{\tnp}{t^{k+1}}
\newcommand{\bb}{{\bf b}}
\newcommand{\fnp}{f^{k+1}}
\newcommand{\prp}{P_R^{'}}
\newcommand{\pr}{P_R}
\newcommand{\rn}{r^k}
\newcommand{\ur}{\bu_r}
\newcommand{\urn}{\bu_r^k}
\newcommand{\utn}{\bu_t^k}
\newcommand{\urnp}{\bu_r^{k+1}}
\newcommand{\utnp}{\bu_t^{k+1}}
\newcommand{\un}{\bu^k}
\newcommand{\unp}{\bu^{k+1}}
\newcommand{\vr}{\bv_r}
\newcommand{\wrn}{\bw_r^k}
\newcommand{\etan}{\eta^k}
\newcommand{\etanp}{\eta^{k+1}}
\newcommand{\prn}{\phi_r^k}
\newcommand{\prnp}{\phi_r^{k+1}}
\newcommand{\prz}{\phi_r^{0}}
\newcommand{\prm}{\phi_{r}^{M}}
\newcommand{\sNn}{\sum\limits_{k=0}^{M-1}}
\newcommand{\sN}{\sum\limits_{k=0}^{M}}
\newcommand{\asNn}{\frac{1}{M-1}\sNn}
\newcommand{\asN}{\frac{1}{M}\sN}
\newcommand{\uph}{\upsilon^h}
\newcommand{\half}{\frac{1}{2}}
\section{Introduction}
\noindent \indent Hamiltonian systems arise in many applications such as mechanics, meteorology and weather prediction, electromagnetism, and modeling of biological oscillators, etc.
In most of the systems,   large-scale computation and long-time integration are required so that preserving intrinsic properties of the system is always desired as an important criterion in the development of numerical schemes.
It is well known that numerical methods such as geometric integrator or structure-preserving algorithms are able to exactly preserve structural properties of Hamiltonian systems.
For Hamiltonian ordinary differential equations (ODEs), development of structure-preserving schemes has achieved a remarkable success, such as the symplectic algorithms proposed in \cite{Hairer02,Feng03}.
In the past two decades, various symplectic algorithms have been extended to Hamiltonian partial differential equations (PDEs) to preserve the multi-symplectic conservation law \cite{Bridges06,Wang2013}.

In recent years, there has been an increasing emphasis on
constructing numerical methods to
preserve certain invariant quantities such as the total energy of continuous dynamical systems.
Various discrete gradient methods have been proposed in the literature, see for example
\cite{Gonzalez96,McLachlan99}.
The averaged vector field (AVF) method was proposed in \cite{Quispel08} for
canonical Hamiltonian systems, in which accurate computation of
integrals are required.
The method is extended in \cite{Hairer10,Cohen11} to arbitrarily high order and non-canonical Hamiltonian
systems.
Based on the AVF method, a systematic energy-preserving or energy dissipation
method is developed in \cite{Celledoni2012}.
The discrete variational derivative method is developed in \cite{Furihata1999} that
inherits energy conservation or dissipation properties for a large class of PDEs.
The method is further generalized in \cite{Matsuo2001,Furihata2011} to complex-valued nonlinear PDEs.
The concept of the discrete variational derivative and a general framework for deriving integral-preserving numerical
methods for PDEs were proposed in \cite{Dahlby2011}.
A class of new structure-preserving methods for multi-symplectic Hamiltonian PDEs was designed in \cite{Gong2014}.

Since many applications of Hamiltonian systems involve repeated, large-scale numerical simulations, reduced-order modeling can be employed to obtain an efficient surrogate model.
One such model reduction technique is the proper orthogonal decomposition (POD) method, which has been successfully applied to many time-dependent, nonlinear PDEs (\cite{bui2007goal,carlberg2011low,chaturantabut2010nonlinear,daescu2008dual,HLB96,iollo2000stability,KV01,sirisup2004spectral,lassila2014model}).
The POD method is a data-driven approach, which extracts a few leading order, orthogonal basis functions from snapshots.
By approximating the state variable in the subspace spanned by this basis set and combining with the Galerkin (or Petrov-Galerkin) methods, one can construct a low-dimensional dynamical system to approximate the original system.
Because the dimension of the reduced-order model (ROM) is low, it is computationally inexpensive for numerical simulations.
Although being successfully applied to many PDE models, when a Hamiltonian system is considered, the POD-ROM is not able to preserve some desired physical quantities because the structure of the original system may not be retained in the dynamical system.

This issue has been discussed recently and attempts are made to address it.
For example, structure-preserving interpolatory projection methods are developed in \cite{beattie2009interpolatory} for keeping structures  such as symmetry of linear dynamical systems. It is then extended to port-Hamiltonian systems via tangential rational interpolation in \cite{gugercin2012structure}.
For port-Hamiltonian systems, structure-preserving Petrov-Galerkin reduced models  are introduced in \cite{beattie2011structure}, in which the POD-based and $\mathcal{H}_2$-based projection subspaces are discussed.
A combination of POD and $\mathcal{H}_2$ quasi-optimal selections of the subspace as well as a structure-preserving nonlinear model reduction via discrete empirical interpolation method (DEIM) is recently developed in \cite{chaturantabut2016structure}.
A proper symplectic decomposition (PSD) approach based on the symplectic Galerkin projection is proposed in \cite{peng2016symplectic} for developing ROMs of the Hamiltonian PDEs with a symplectic structure.
Three algorithms are designed to extract the PSD basis from snapshots computed.
There are some other model reduction approaches that aim to preserve Lagrangian structures, for instance, in \cite{carlberg2015preserving,farhat2015structure}.

In this paper, we use the standard framework of Galerkin projection-based model reduction, but propose a structure-preserving ROM of Hamiltonian systems.
The new ROM possesses the appropriate Hamiltonian structure; however, due to the POD truncation, there may exist discrepancies in the approximation to the Hamiltonian function between the new ROM and the original or the full-order model (FOM).
We then introduce the POD basis from shifted snapshots to improve the reduced-order approximation.

The rest of this paper is organized as follows.
In Section \ref{sec: alg}, we introduce the structure-preserving Galerkin (SPG) POD reduced-order modeling approach for Hamiltonian systems;
Several issues on the implementation are discussed in Section \ref{sec:imp};
An {\em a priori} error estimate is presented in Section \ref{sec:err};
The effectiveness of the proposed ROM is demonstrated through several numerical examples in Section \ref{sec: num};
A few concluding remarks are drawn in the last section.

\section{Structure-Preserving Galerkin ROMs}\label{sec: alg}

\noindent \indent We consider a general Hamiltonian PDE system
\begin{equation}
\dot{\bu} = \mathcal{D}\, \frac{\delta \mathcal{H}}{\delta \bu},
\label{eq: ham_pde}
\end{equation}
where $\mathcal{H}(\bu)$ is the Hamiltonian of the system   which often corresponds to its total energy.
It preserves $\mathcal{H}(\bu)$ invariant if $\mathcal{D}$ is a skew-adjoint operator.

Due to the need for long-term integrations, numerical methods for solving the Hamiltonian PDE system are expected to possess certain stability and preservation properties.
Based on the property of the Hamiltonian system, keeping the energy invariant is usually used as a basic guiding principle while designing a robust numerical scheme.
For the purpose of reducing the computational cost, the POD-ROM can be considered in repeated, large-scale numerical simulations of the Hamiltonian system.
However, the standard Galerkin projection-based model reduction technique may destroy the structure of the reduced order system.
In this section, we develop a new POD-ROM for the Hamiltonian system, which uses the Galerkin projection-based model reduction technique, but is modified to retain the appropriate structure of the PDE system.
In the following, we briefly introduce the POD method (for details on POD, the reader is referred to \cite{HLB96}).


\subsection{The proper orthogonal decomposition approximation}

\noindent \indent The POD method extracts essential information from some snapshot data of a system, which is then used to form a global basis.
The original state variable is then approximated in the space spanned by only a handful of dominant basis functions.

Suppose the system is numerically discretized by using certain methods, in which the number of degrees of freedom (DOF) in space is $n$.
Let $\by_i$ represent the numerical solution at time instance $t_i$.
The snapshot data $\bY$ consists of numerical solutions at selected time instances $t_1, \ldots, t_m$, i.e., $\bY=[\by_1, \ldots, \by_m]\in \mathbb{R}^{n\times m}$.
The POD method seeks an orthonormal basis $\{\bphi_1, \ldots, \bphi_r\}$ from the optimization problem
\begin{equation*}
\min_{Rank(\bPhi)= r} \sum_{j=1}^M
  \Big\| \by_j - \bPhi \bPhi^\top \by_j \Big\|^2
  \quad s. t. \quad
  \bPhi^\top\bPhi= I_r,
 \label{eq:pod_h}
\end{equation*}
where $\bPhi= [\bphi_1, \ldots, \bphi_r]\in \mathbb{R}^{n\times r}$, $I_r$ is the $r\times r$ identity matrix, and
$\|\cdot\|$ denotes the 2-norm in the Euclidean space throughout the paper.

It is well-known that the POD basis vectors, $\bphi_1, \ldots, \bphi_r$, are the left singular vectors of $\bY$ corresponding to the first $r$ leading nonzero singular values $\sigma_1\geq \sigma_2\geq \ldots\geq \sigma_r>0$. When $\bY$ is low-and-fat ($n<<m$), the POD basis can be found by using the SVD algorithm. Conversely, if $\bY$ is tall-and-skinny ($n>>m$), the method of snapshots can be applied.
When parallel computing is implemented in large-scale computing problems, one can use an approximate partitioned method of snapshots to further reduce the computational complexity and the communication volume for generating the POD basis \cite{Wang2015approximate}.

\begin{prop}
\label{pr:poderr_dis}
Let $\bPhi$ be the $r$-dimensional POD basis, the rank of snapshot matrix $\bY$ be $d$ and $\sigma_1\geq \sigma_2\geq \ldots\geq \sigma_d>0$ be nonzero singular values of $\bY$.
Then the POD projection error of the snapshot matrix satisfies (\cite{KV01,chaturantabut2012state})
\begin{equation}
\sum_{j=1}^M
  \Big\| \by_j - \bPhi \bPhi^\top \by_j \Big\|^2
  = \sum\limits_{j=r+1}^d \sigma_j^2.
 \label{eq:pod_herr}
\end{equation}
\end{prop}

We shall also consider an ideal case, in which the entire continuous trajectory of the system is assumed to be available on the whole time interval $[0, T]$.
Taken the trajectory $\by(t)$ to be the snapshot, which is a continuous, vector-valued function in $\mathbb{R}^n$.
The POD method is to find a set of optimal basis functions $\{\bphi_1, \ldots, \bphi_r\}$ by minimizing the projection error of $\by(t)$ onto the subspace spanned by the basis, i.e.,
\begin{equation*}
\min_{Rank(\bPhi)= r}\int_0^T
  \Big\| \by(t) - \bPhi \bPhi^\top \by(t) \Big\|^2 \, d t
  \quad s. t. \quad
  \bPhi^\top\bPhi= I_r,
 \label{eq:pod}
\end{equation*}
where $\bPhi= [\bphi_1, \ldots, \bphi_r]\in \mathbb{R}^{n\times r}$ and $I_r$ is the identity matrix.
The optimization solution is equivalent to find the first $r$ dominant eigenvectors of the snapshot covariant matrix ${\bf R}=\int_0^T \by(t)\by(t)^\top\, dt$.
\begin{prop}
\label{pr:poderr}
Let $\bPhi$ be the $r$-dimensional POD basis, the rank of snapshot covariant matrix ${\bf R}$ be $d$ and $\lambda_1\geq \lambda_2\geq \ldots\geq \lambda_d>0$ be nonzero eigenvalues of ${\bf R}$.
Then the POD projection error of the snapshot satisfies (\cite{KV01,chaturantabut2012state})
\begin{equation}
\int_0^T
  \Big\| \by(t) - \bPhi \bPhi^\top \by(t) \Big\|^2 \, d t
  = \sum\limits_{j=r+1}^d \lambda_j.
 \label{eq:poderr}
\end{equation}
\end{prop}


\subsection{The structure-preserving POD-ROM}
\noindent \indent After a spatial discretization of (\ref{eq: ham_pde}), the finite dimensional Hamiltonian ODE system is given by  (e.g. see \cite{Celledoni2012})
\begin{equation}
\dot{\bu} = \bD\, \nabla_{\bu} H(\bu),
\label{eq: ham_fom}
\end{equation}
where $\bD$ is a skew-symmetric matrix.
The system is complemented by an initial condition $\bu(t_0)= \bu_0$ and appropriate boundary conditions.
Note that we abuse the notation here by still using $\bu$ to denote the discrete variable $\bu\in {\cal R}^n$, where $n$ is the dimension of the discrete variable.

Suppose the POD basis matrix $\bPhi$ is obtained, the reduced-order approximation of the state variable $\bu$ is $\bu^r = \bPhi \ba(t)$, where $\ba(t)$ is the unknown coefficient vector.
Substituting $\bu$ with $\bu^r$ in \eqref{eq: ham_fom}, we have
\begin{equation*}
\Phi\dot{\ba} = \bD\, \nabla_{\bu} H(\bPhi \ba).
\label{eq: ham_rom}
\end{equation*}
Applying the Galerkin method by multiplying $\bPhi^\top$ on both sides and using the fact $\bPhi^\top \bPhi= I_r$, we have
\begin{equation}
\dot{\ba} = \bPhi^\top\bD\, \nabla_{\bu} H(\bPhi \ba).
\label{eq: ham_rom1}
\end{equation}
The time derivative of Hamiltonian function $H$ (usually, energy of the system) is
\begin{eqnarray*}
\frac{d}{dt}H(\bPhi \ba)&=&[\nabla_\ba H(\bPhi \ba)]^\top \dot{\ba} \nonumber \\
				  &=&[\bPhi^\top \nabla_{\bu} H(\bPhi \ba)]^\top \bPhi^\top\bD\, \nabla_{\bu} H(\bPhi \ba) \nonumber \\
				  &=& \nabla_{\bu} H(\bPhi \ba)^\top \bPhi \bPhi^\top\bD\, \nabla_{\bu} H(\bPhi \ba),
\end{eqnarray*}
where we use the fact that
\begin{equation}
\nabla_{\ba} H(\bPhi \ba) = \bPhi^\top \nabla_{\bu} H(\bPhi \ba).
\label{eq: ham_grad}
\end{equation}
Recall that $\bPhi$ is composed of the dominant $r$ left singular vectors of the snapshot matrix.
In the case of $r= n$, $\bPhi$ is a square unitary matrix and $\bPhi\bPhi^\top= I_n$, we have $\frac{d}{dt}H(\bPhi \ba)= 0$ because $\bD$ is skew-symmetric, then the Hamiltonian is a constant.
However, $r\ll n$ in most cases and $\bPhi$ is not a square matrix.
Thus, $\bPhi\bPhi^\top\bD$ does not have the property of $\bD$  as a skew-symmetric matrix so that the property of the Hamiltonian function is not guaranteed to be preserved.
Next, we follow the same framework of Galerkin projection-based POD reduced order modeling, but modify the model so that the Hamiltonian structure is well kept after applying ROM.





\paragraph{Structure-preserving ROMs}
In order to keep the structure of the Hamiltonian system, we assume that there exists a matrix $\bD_r$ with the same property like $\bD$ such that, in \eqref{eq: ham_rom1},
\begin{equation}
\bPhi^\top\bD = \bD_r\bPhi^\top.
\label{eq: Sr0}
\end{equation}
With the new $\bD_r$, we have the Galerkin projection-based ROM in the following from:
\begin{eqnarray}
\dot{{\ba}} &=& \bPhi^\top \bD \, \nabla_{\bu} H(\bPhi {\ba}) \nonumber\\
			    &=& \bD_r\bPhi^\top \, \nabla_{\bu} H(\bPhi {\ba})\nonumber\\
&= & \bD_r\, \nabla_{\ba} H(\bPhi {\ba})
\label{eq: ham_rom2}
\end{eqnarray}
with initial condition $ {\ba}(t_0)= \bPhi^\top\bu_0$.
In this reduced-order dynamical system, we have the time derivative of the Hamiltonian function $H$ is
\begin{eqnarray*}
\frac{d}{dt}H(\bPhi  {\ba})&=&[\nabla_{ {\ba}} H(\bPhi {\ba})]^\top \dot{ {\ba}} \nonumber \\
				  &=&[\nabla_{ {\ba}} H(\bPhi  {\ba})]^\top \bD_r\, \nabla_{\ba} H(\bPhi  {\ba}) \nonumber \\
				  &=& 0,\quad \text{ if $\bD_r$ is skew-symmetric.}
\end{eqnarray*}
Therefore, the reduced-order model possesses an invariant Hamiltonian, in which the key is how to obtain the new skew symmetric matrix $\bD_r$.

Note that equation \eqref{eq: Sr0} is over-determined for $\bD_r$.
By right multiplying $\bPhi$ on both sides of the equation, we have the normal equation solution
\begin{equation}
\bD_r = \bPhi^\top \bD \bPhi,
\label{eq: Sr}
\end{equation}
which is skew symmetric if $\bD$ is skew symmetric.
Therefore, when an appropriate numerical algorithm is employed, the structure-preserving ROM (SP-ROM) \eqref{eq: ham_rom2}-\eqref{eq: Sr} will keep the Hamiltonian function approximation  invariant as that in the FOM.
We will discuss several implementation issues of the model and present a theoretical analysis in next sections.

\begin{remark}
When the operator $\mathcal{D}$ in equation \eqref{eq: ham_pde} is negative (semi-) definite, the system possesses a non-increasing Lyapunov function $H(\bu)$.
The corresponding dynamical system after a spatial discretization has a negative (semi-) definite coefficient matrix $\bD$.
The standard POD-G ROM doesn't preserve the structure since $\bPhi\bPhi^\top\bD$ is not negative definite in general.
The proposed structure-preserving ROM can still be used to circumvent this problem.
Indeed, by choosing $\bD_r= \bPhi^\top \bD \bPhi$ in \eqref{eq: ham_rom2}, the resulting ROM  guarantees that
$\frac{d}{dt}H(\bPhi {\ba})=[\nabla_{{\ba}} H(\bPhi {\ba})]^\top \bD_r\, \nabla_{\ba} H(\bPhi {\ba})\leq 0$.
Thus, the Lyapunov approximation $H(\bPhi {\ba})$ is non-increasing.
 \end{remark}

\section{Choices on implementations\label{sec:imp}}
\noindent \indent In this section, we describe some details on the implementation.
\subsection{Choice of snapshots}
\noindent \indent Motivated by the error estimate in Section \ref{sec:err}, besides the collections of state variables $\bu(t_j)$ for $j=1, \ldots, M$, we also include in the snapshots weighted gradient information of the Hamiltonian function, $\mu \nabla_{\bu}H(\bu(t_j))$.
Denote the snapshot matrix by
\begin{equation}
\bY=\left[\bu(t_1), \ldots, \bu(t_M), \mu\nabla_{\bu}H(\bu(t_1)), \ldots, \mu\nabla_{\bu}H(\bu(t_M))\right].
\label{eq: snap}
\end{equation}
The POD basis $\bPhi$ functions are the dominant $r$ left singular vectors of the snapshot matrix.

Correspondingly, based on Proposition \ref{pr:poderr}, we have the following POD projection errors under the assumption on the availability of the entire continuous trajectory:
\begin{equation}
\int_0^T
  \Big\| \bu(t) - \bPhi \bPhi^\top \bu(t) \Big\|^2 \, d t
 +
  \mu^2 \int_0^T
 \Big\| \nabla_{\bu}H(\bu(t)) - \bPhi \bPhi^\top \nabla_{\bu}H(\bu(t)) \Big\|^2 \, d t
  = \sum\limits_{j=r+1}^d \lambda_j,
 \label{eq:poderr_h}
\end{equation}
where $\lambda_j$ is the $j$-th leading eigenvalue of the snapshot covariance matrix.

\subsection{Improvement of the Hamiltonian approximation}
Although a few POD basis functions capture most of the system information, the basis truncation would result in a loss of information, which leads to a discrepancy of the Hamiltonian function approximation between the ROM and the FOM.
Since the proposed SP-ROM is able to keep a constant Hamiltonian in the reduced-order simulation, for improving the Hamiltonian function approximation, we only need to adjust its initial value.
To this end, we propose a way to address this issue and compare it with the approach in \cite{peng2016symplectic}.

\paragraph{I. Enrichment of the POD basis (\cite{peng2016symplectic})} It is natural to introduce a new basis that could capture the missing information of the POD approximation at the beginning of the simulation.

The POD projection of initial data $\bu_0$ is given by $P_{\bPhi}\bu_0= \bPhi\bPhi^\top\bu_0$ and the related residual is $\br= ({\bf I}- P_{\bPhi})\bu_0$, which is orthogonal to the POD basis.
When $\br$ is nonzero, it can be normalized and added to the original POD basis set as a new basis function $\bpsi=\frac{\br}{\|\br\|}$.
Obviously, when $\br$ is zero, no action is needed.

Using the enriched basis guarantees the initial information be completely captured. 
Together with the SP-ROM, the invariant Hamiltonian function would be well kept in the reduced-order simulation.
However, the new basis function $\bpsi$ aims to capture initial information, but it will be used as a global basis in the entire simulation; thus, it may degrade the overall accuracy of the state variable approximation.

\paragraph{II. POD basis of shifted snapshots}
We propose a new way to overcome the Hamiltonian discrepancy by using the POD basis from shifted snapshots.
Instead of considering $\bY$ in \eqref{eq: snap}, we use a shifted snapshot set
\begin{equation}
\left[\bu(t_1)-\bu_0, \ldots, \bu(t_M)-\bu_0,
\mu\nabla_{\bu} H(\bu(t_1), \ldots, \mu\nabla_{\bu} H(\bu(t_M)\right].
\label{eq: snap_shift}
\end{equation}
The related POD approximation becomes
\begin{equation}
\bu= \bu_0+\bPhi {\ba}(t),
\label{eq: pod_shift}
\end{equation}
where $\bPhi$ comes from the shifted snapshots and ${\ba}$ is the unknown coefficient.
Substituting \eqref{eq: pod_shift} into the SP-ROM \eqref{eq: ham_rom2}, we have
\begin{equation}
\dot{{\ba}} = \bD_r\, \nabla_{{\ba}} H(\bu_0+\bPhi {\ba}),
\label{eq: ham_rom3}
\end{equation}
where $\bD_r = \bPhi^\top \bD \bPhi$ with the initial condition ${\ba}(t_0) = \bf{0}$.

\section{Error estimate\label{sec:err}}
\noindent \indent In this section, we present an {\em a priori} error estimation of the SP-ROM \eqref{eq: ham_rom2}-\eqref{eq: Sr}.
The derivation mainly follows the state space error estimation developed by Chaturantabut and Sorensen in \cite{chaturantabut2012state} for the POD-DEIM nonlinear model reduction.
The difference lies in the structure of proposed SP-ROM and our choice of snapshots.
In a sequel, we analyze the approximation error of the reduced-order simulation and we shall focus on the error caused by the POD truncation by suppressing the spatial and temporal discretization errors.

Define the Lipschitz constant and logarithmic Lipschitz constant of a mapping $F: \mathbb{R}^n\rightarrow \mathbb{R}^n$, respectively, as follows:
$$
\mathcal{L}[F]= \sup_{\bu\neq \bv}\frac{\|F(\bu)-F(\bv)\|}{\|\bu-\bv\|},
\quad \text{and} \quad
\mathcal{M}[F]= \sup_{\bu\neq \bv}\frac{\left<\bu-\bv, F(\bu)-F(\bv)\right>}{\|\bu-\bv\|^2},
$$
where the Euclidian inner product $\left<\cdot, \cdot\right>: \mathbb{R}^d\times \mathbb{R}^d \rightarrow \mathbb{R}$ is for any positive integer $d$.
The logarithmic Lipschitz constant could be negative, which was used in \cite{chaturantabut2012state} to show that the error of reduced-order solution is uniformly bounded on $t\in [0, T]$ when the map $F$ is uniformly negative monotonic.


\begin{thm}
Let $\bu(t)$ be the solution of the FOM \eqref{eq: ham_fom}, ${\ba}(t)$ be the solution of the SP-ROM  \eqref{eq: ham_rom2}-\eqref{eq: Sr} with the initial condition ${\ba}(t_0)= \bPhi^\top \bu(\cdot, 0)$, the POD approximation error satisfies
$$
\int_0^T \|\bu(t)-\bPhi{\ba}(t)\|^2\, dt \leq C(T, \mu)\sum\limits_{j=r+1}^{d} \lambda_j,
$$
where $\lambda_j$ is the $j$-th leading eigenvalue of the snapshot covariance matrix associated to the snapshots defined in \eqref{eq: snap}.
\label{thm: 1}
\end{thm}
\begin{proof}
Define the POD approximation error
\begin{equation}
\be= \bu-\bPhi {{\ba}} = \bu-\bPhi\bPhi^\top \bu+\bPhi\bPhi^\top \bu-\bPhi {{\ba}}= \rho+\theta,
\label{eq:e0}
\end{equation}
where $\rho= \bu-\bPhi\bPhi^\top \bu$ and $\theta= \bPhi\bPhi^\top \bu-\bPhi {\ba}$.
%
Note that
\begin{equation}
\frac{d}{dt}\|\theta\| = \frac{1}{2\|{\theta}\|}\frac{d}{dt}\|{\theta}\|^2= \frac{1}{\|{\theta}\|}\left<{\theta},\, \dot{{\theta}}\right>
\label{eq:d_theta}
\end{equation}
and
\begin{eqnarray}
\dot{\theta} &=& \bPhi \bPhi^\top \dot{\bu} - \bPhi \dot{{\ba}} \nonumber \\
		  &=& \bPhi \bPhi^\top \bD\, \nabla_{\bu} H(\bu)  - \bPhi \bD_r\, \nabla_{{\ba}} H(\bPhi {\ba}) \nonumber \\
		  &=& \bPhi \bPhi^\top \bD\, \nabla_{\bu} H(\bu)  - \bPhi \bPhi^\top \bD\bPhi \bPhi^\top\, \nabla_{\bu} H(\bu) \nonumber \\
		  &&+ \bPhi \bPhi^\top \bD\bPhi \bPhi^\top\, \nabla_{\bu} H(\bu) - \bPhi \bPhi^\top \bD\bPhi \bPhi^\top\, \nabla_{\bu} H(\bPhi \bPhi^\top\bu) \nonumber \\
		  &&+\bPhi \bPhi^\top \bD\bPhi \bPhi^\top\, \nabla_{\bu} H(\bPhi \bPhi^\top\bu)-\bPhi \bPhi^\top \bD  \bPhi \bPhi^\top\nabla_{\bu} H(\bPhi {\ba}).
\label{eq:d_theta00}		
\end{eqnarray}
Testing equation \eqref{eq:d_theta00} by $\theta$, we have, on the right-hand side of \eqref{eq:d_theta00},
\begin{equation}
\left<{\theta}, \bPhi \bPhi^\top \bD\, (I-\bPhi \bPhi^\top)\nabla_{\bu} H(\bu) \right>
 \leq \|\theta\|\, \|\bPhi \bPhi^\top \bD \|\, \|(I-\bPhi \bPhi^\top)\nabla_{\bu} H(\bu)\|,
 \label{eq:err_projH}
\end{equation}
\begin{equation}
\left<{\theta}, \bPhi \bPhi^\top \bD\bPhi \bPhi^\top\, (\nabla_{\bu} H(\bu) - \nabla_{\bu} H(\bPhi \bPhi^\top\bu)) \right>
\leq \|\theta\|\, \|\bPhi \bPhi^\top \bD \bPhi \bPhi^\top\|\, \mathcal{L}[\nabla_{\bu}H]\, \|(I-\bPhi \bPhi^\top)\bu\|,
 \label{eq:err_proju}
\end{equation}
\begin{equation}
\left<{\theta}, \bPhi \bPhi^\top \bD\bPhi \bPhi^\top\, (\nabla_{\bu} H(\bPhi \bPhi^\top\bu) - \nabla_{\bu} H(\bPhi {\ba})) \right>
\leq \mathcal{M}[ \bPhi \bPhi^\top \bD\bPhi \bPhi^\top\nabla_{\bu}H]\, \|\theta\|^2.
 \label{eq:err_MF}
\end{equation}
Let $C_1= \mathcal{M}[ \bPhi \bPhi^\top \bD\bPhi \bPhi^\top\nabla_{\bu}H]$,
$C_2= \|\bPhi \bPhi^\top \bD \bPhi \bPhi^\top\|\, \mathcal{L}[\nabla_{\bu}H]$,
and $C_3= \|\bPhi \bPhi^\top \bD \|$ and combine \eqref{eq:d_theta} with \eqref{eq:d_theta00}-\eqref{eq:err_MF}, we have
\begin{equation*}
\frac{d}{dt}\|\theta\| \leq C_1 \|\theta\| + C_2 \|\rho\| + C_3 \|\eta\|,
\label{eq:d_theta0}
\end{equation*}
where $\eta= \nabla_{\bu} H(\bu)-\bPhi\bPhi^\top \nabla_{\bu} H(\bu)$.
Applying Gronwall's lemma, for any $t\in [0, T]$, we get
\begin{equation*}
\|\theta(t)\| \leq \int_0^t e^{C_1 (t-\tau)} (C_2 \|\rho\| + C_3 \|\eta\|)\, d\tau,
\label{eq:d_theta1}
\end{equation*}
where we used the fact that $\theta(0)=0$.
Therefore, for any $t\in [0, T]$, we have
\begin{equation*}
\|\theta(t)\|^2 \leq \alpha(T)\left[C_2^2 \int_0^T \|\rho\|^2 dt+ C_3^2 \int_0^T \|\eta\|^2\, dt\right],
\label{eq:d_theta2}
\end{equation*}
where $\alpha(T)= 2\int_0^T e^{2C_1 (T-\tau)}\, d\tau$.
Hence,
\begin{equation}
\int_0^T\|\theta(t)\|^2\, dt \leq T \alpha(T)\left[C_2^2 \int_0^T \|\rho\|^2 dt+ C_3^2 \int_0^T \|\eta\|^2\, dt\right].
\label{eq:d_theta3}
\end{equation}
A combination with \eqref{eq:poderr_h} and \eqref{eq:e0} indicates that the POD approximation error is bounded by
$$
\int_0^T \|\be(t)\|^2\, dt \leq C(T, \mu)\sum\limits_{j>r} \lambda_j,
$$
where $C(T, \mu)= 1+T\alpha(T)(C_2^2+C_3^2/\mu^2)$.
\end{proof}

\begin{thm}
Let $\bu(t)$ be the solution of the FOM \eqref{eq: ham_fom} and ${\ba}(t)$ be the solution of the SP-ROM  \eqref{eq: ham_rom3} and \eqref{eq: Sr} with the initial condition ${\ba}(t_0)= {\bf 0}$, the POD approximation error satisfies
$$
\int_0^T \|\bu(t)-(\bu_0+\bPhi{\ba}(t))\|^2\, dt \leq C(T, \mu)\sum\limits_{j=r+1}^{d} \lambda_j,
$$
where $\lambda_j$ is the $j$-th leading eigenvalue of the snapshot covariance matrix associated to the shifted snapshots defined in \eqref{eq: snap_shift}.
\end{thm}
\begin{proof}
Define $\bw= \bu-\bu_0$, we rewrite the FOM \eqref{eq: ham_fom} in terms of $\bw$ as follows.
\begin{equation}
\dot{\bw} = \bD\, \nabla_{\bu} H(\bu_0+\bw)
\label{eq: ham_fom_v}
\end{equation}
with initial condition $\bw(t_0)= 0$.
The approximation error of the SP-ROM \eqref{eq: ham_rom3} and \eqref{eq: Sr} is
\begin{equation}
\be= (\bu_0+\bw)-(\bu_0+\bPhi {{\ba}})= \bw-\bPhi {{\ba}},
\end{equation}
which can be analyzed in the same manner as that is done in Theorem \ref{thm: 1}.
\end{proof}

\section{Numerical experiments\label{sec: num}}
\noindent \indent In this section, we consider the following Hamiltonian PDEs: 1) the wave equation; 2) the Korteweg-de Vries (KdV) equation.
Both of them are of certain structures and have a constant energy, in particular, the first one has the symplectic structure.

\subsection{Wave equation\label{sec: we}}
\noindent \indent Consider the one-dimensional semilinear wave equation with a constant moving speed $c$ and a nonlinear forcing term $g(u)$,
\begin{equation*}
u_{tt}= c^2 u_{xx} - g(u), \quad 0\leq x\leq l.
\end{equation*}
The equation can be written in the Hamiltonian formulation, which has a symplectic structure, as follows:
\begin{equation}
\left[
\begin{array}{c}
\dot{u}\\
\dot{v}
\end{array}
\right] =
\left[
\begin{array}{cc}
0 & 1 \\
-1 & 0
\end{array}
\right]
\left[
\begin{array}{c}
\frac{\delta H}{\delta u}\\
\frac{\delta H}{\delta v}
\end{array}
\right],
\label{eq:lin_wave}
\end{equation}
where the Hamiltonian, also the system energy,
\begin{equation*}
H(u, v) = \int_0^l \left[\frac{1}{2}v^2+\frac{c^2}{2} u_x^2 + G(u)\right]\, dx,
\end{equation*}
with $G'(u) = g(u)$,
$\frac{\delta H}{\delta u} = -c^2 u_{xx}+g(u)$ and $\frac{\delta H}{\delta v} = v$.
After a spatial discretization with the number of degrees of freedom $n$, \eqref{eq:lin_wave} is rewritten into
\begin{equation}
\left[
\begin{array}{c}
\dot{\bu}\\
\dot{\bv}
\end{array}
\right] =
\left[
\begin{array}{cc}
0 & \bI_n \\
-\bI_n & 0
\end{array}
\right]
\left[
\begin{array}{c}
-\bA \bu+\mathcal{G}'(\bu)\\
\bv
\end{array}
\right],
\end{equation}
where $\bA$ is a discrete, scaled, one-dimensional second order differential operator and $\mathcal{G}({\bu})$ is the discretization of the nonlinear function $G({\bu})$.
The coefficient matrix is  skew-symmetric.
The problem has been tested by a proper symplectic decomposition method in \cite{peng2016symplectic}.
We will use the same example to test our structure preserving ROMs.

Consider only a linear case $g({\bu})=0$ since nonlinearity doesn't affect the structure of ROMs.
In particular, we use $c= 0.1$, $x\in [0, 1]$, $t\in [0, 50]$, and periodic boundary conditions.
The initial condition satisfies $u(0)=h(s(x))$ and $\dot{u}(0)=0$, where $h(s)$ is a cubic spline function defined by
\begin{equation*}
h(s) =
\left\{
\begin{array}{ll}
1-\frac{3}{2}s^2+\frac{3}{4}s^3 & \text{\quad if  \,\,} 0\leq s\leq 1, \\
\frac{1}{4}(2-s)^3 		       & \text{\quad if  \,\,} 1< s\leq 2, \\
0					       & \text{\quad if  \,\,} s> 2,
\end{array}
\right.
\end{equation*}
and $s(x)= 10|x-\frac{1}{2}|$.

In a full-order simulation, the spatial domain is partitioned into $n=500$ equal subintervals, thus the mesh size $\Delta x= 2\times 10^{-3}$.
We use the average vector field method (AVF) for the time integration with a time step $\Delta t= 0.01$.
A three-point stencil finite difference method is taken for the spatial discretization of the 1D Laplacian operator.
The fully discrete scheme solution at time $t_{k+1}$, $\bu_h^{k+1}$ and $\bv_h^{k+1}$, satisfies
\begin{equation}
\left[
\begin{array}{c}
\frac{\bu_h^{k+1}-\bu_h^k}{\Delta t}\\
\frac{\bv_h^{k+1}-\bv_h^k}{\Delta t}
\end{array}
\right] =
\left[
\begin{array}{cc}
0 &\bI_n \\
-\bI_n & 0
\end{array}
\right]
\left[
\begin{array}{c}
-\bA\frac{\bu_h^{k+1}+\bu_h^k}{2}\\
\frac{\bv_h^{k+1}+\bv_h^k}{2}
\end{array}
\right],
\end{equation}
where
\begin{equation*}
\bA= \frac{c^2}{\Delta x^2}
\left( \begin{array}{cccccc}
-2 & 1 & 0  & 0 & \cdots & 1 \\
1 & -2 & 1 & 0 & \cdots & 0 \\
   &     & \ddots & \ddots & \ddots &  \\
 0   &    \cdots&0          &  1 & -2  & 1 \\
 1 &    \cdots      &0     &  0  &  1 & -2
\end{array} \right),
\end{equation*}
the initial data $\bu_h^0$ has the $i$-th component equals $h(s(x_i))$ for $1\leq i\leq n$ and $\bv_h^0= {\bf 0}$.
The linear system is solved by the built-in direct solver of Matlab software. 
The time evolution of $u(x, t)$, $v(x, t)$ and energy $H(t)$ of the full-oder results are plotted in Figure \ref{Fig: lin_wave_full}.
It is seen that the energy is accurately preserved at t./Figures/he value $H= 7.5\times 10^{-2}$.
\begin{figure}[htb]
\centering
\begin{minipage}[ht]{0.31\linewidth}
\includegraphics[width=1\textwidth]{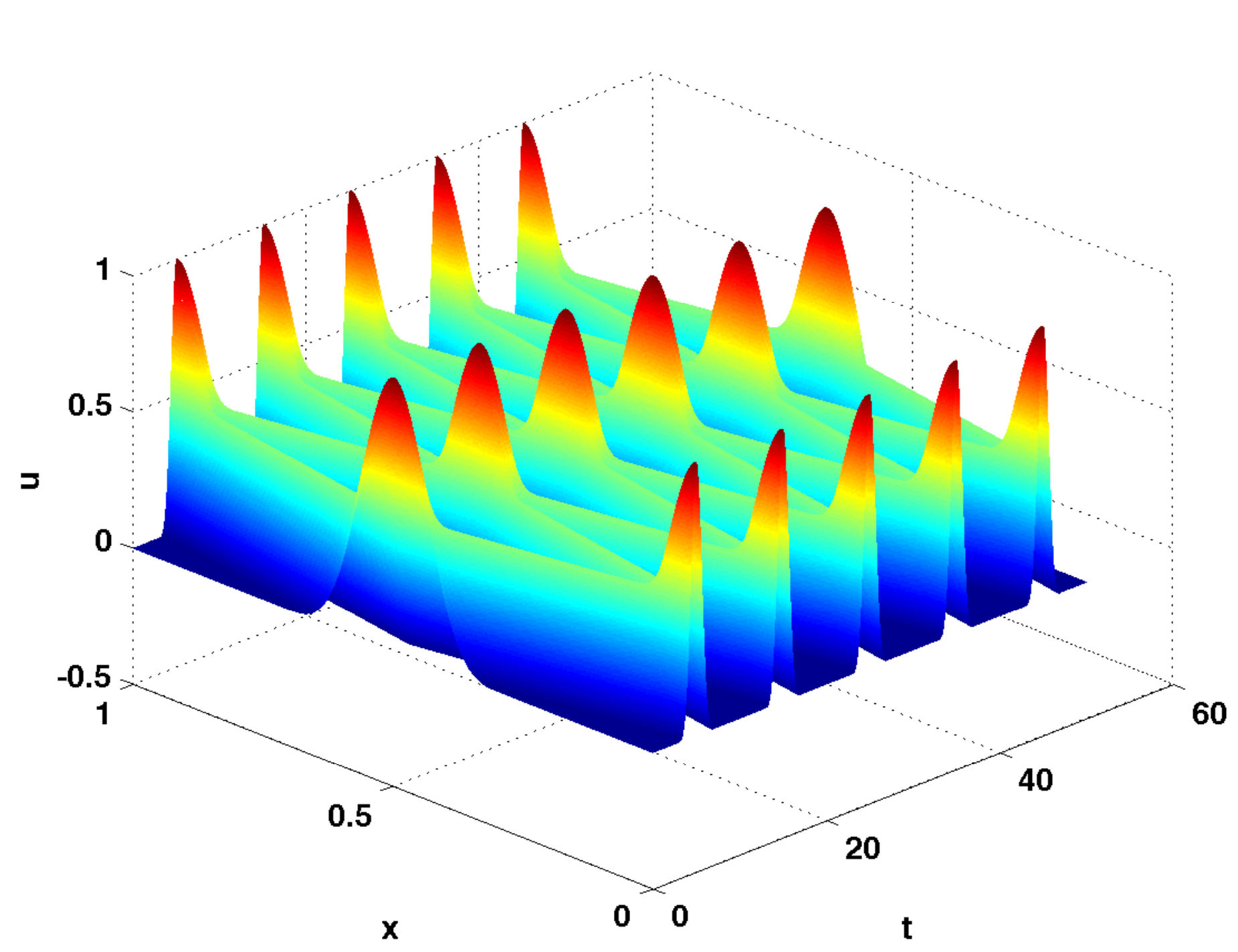}
\end{minipage}
\begin{minipage}[ht]{0.31\linewidth}
\includegraphics[width=1\textwidth]{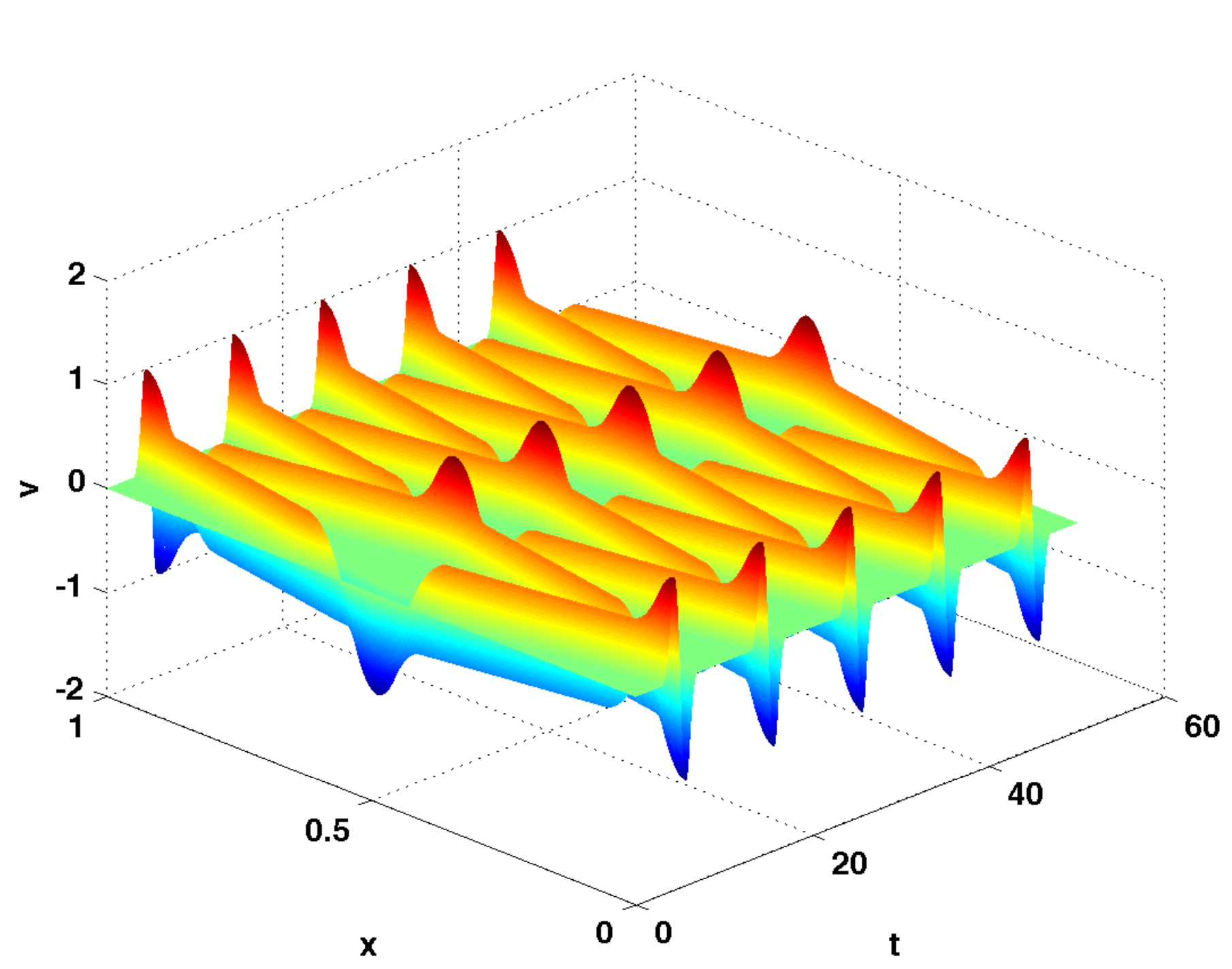}
\end{minipage}
\begin{minipage}[ht]{0.31\linewidth}
\includegraphics[width=1\textwidth]{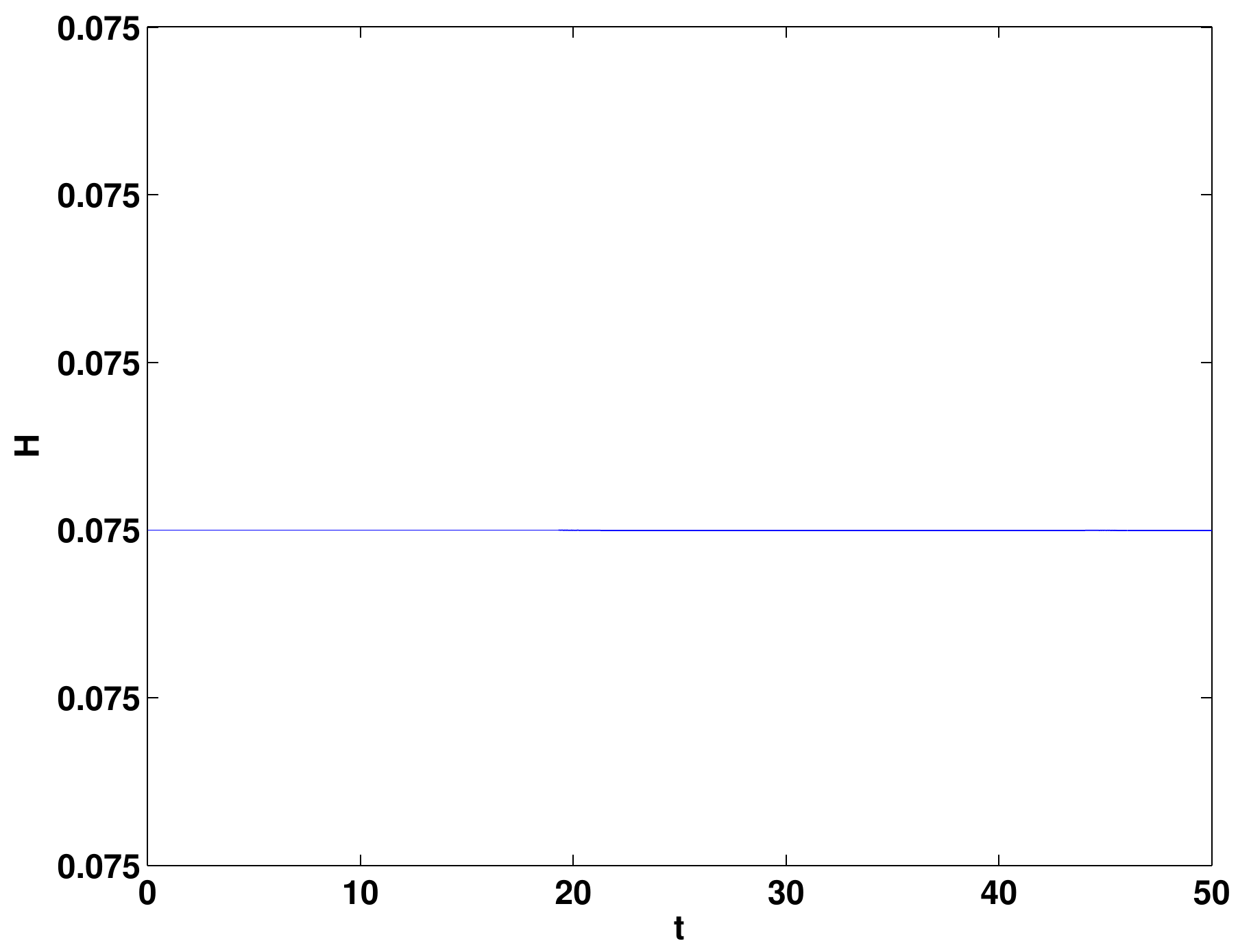}
\end{minipage}
\caption{
Full-order model simulation: time evolution of $u$ (left), time evolution of $v$ (middle), and energy evolution $H(t)$ (right).
}\label{Fig: lin_wave_full}
\end{figure}
Since the exact solution is unknown, in what follows, the full-order simulation results will be served as our benchmark solution.

Next, we investigate the numerical performance of four ROMs: i) standard POD-G ROM; ii) the SP-ROM with standard POD basis; iii) the SP-ROM with enriched POD basis; iv) the SP-ROM with POD basis of shifted snapshots.
The criteria we shall use include: the maximum approximation error over the entire spatial-temporal domain,
$\mathcal{E}_\infty= \max\limits_{k\geq 0} \max\limits_{0\leq i\leq n}  \sqrt{\left[(\bu_h^k)_i-(\bu_r^k)_i\right]^2+\left[(\bv_h^k)_i-(\bv_r^k)_i\right]^2}$,
and the energy value in the reduced-order simulation $H_r(t)$.

\subsubsection{Standard POD-G ROM.}
\noindent \indent Snapshots are collected from the full-order simulation every $50$ time steps.
Using the singular value decomposition, we find the $r$-dimensional POD basis $\bPhi_u$ and $\bPhi_v$ for functions $u$ and $v$, respectively.
The standard POD-G ROM is generated as follows by substituting the POD approximation $\bu_r(t) = \bPhi_u \ba(t)$ and $\bv_r(t) = \bPhi_v \bb(t)$ into the FOM and applying the Galerkin projection.
The model reads
\begin{equation}
\left[
\begin{array}{c}
\dot{\ba}\\
\dot{\bb}
\end{array}
\right] =
\left[
\begin{array}{c}
\bPhi_u^\top\bPhi_v \bb\\
 \bPhi_v^\top\bA \bPhi_u \ba
\end{array}
\right].
\label{eq:lin_wave_rom}
\end{equation}
Under the same time integration method as the FOM, we have the POD basis coefficient at $t_{k+1}$, $\ba^{k+1}$ and $\bb^{k+1}$, satisfying
\begin{equation}
\left[
\begin{array}{c}
\frac{\ba^{k+1}-\ba^{k}}{\Delta t}\\
\frac{\bb^{k+1}-\bb^{k}}{\Delta t}
\end{array}
\right] =
\left[
\begin{array}{c}
\bPhi_u^\top\bPhi_v \frac{\bb^{k+1}+\bb^{k}}{2}\\
\bPhi_v^\top\bA \bPhi_u \frac{\ba^{k+1}+\ba^{k}}{2}
\end{array}
\right]
\label{eq:lin_wave_rom_dis}
\end{equation}
with $\ba^0 = \bPhi_u^\top \bu_0$ and $\bb^0 = \bPhi_v^\top \bv_0$.

When $r=5$, the maximum error of the reduced-order simulation is $\mathcal{E}_\infty= 0.4591$.
In particular, it is observed that the Hamiltonian function varies with time as shown in Figure \ref{Fig: lin_wave_red20_stand} (left), which is an evidence of the standard POD-G ROM is not structure preserving.
Indeed, the value of Hamiltonian function at the final time is 17.43\% larger than the accurate value.

When the dimension increases to $r= 20$,
the POD-G ROM results are improved: maximum approximation error is $\mathcal{E}_\infty= 0.0208$;
the approximation error of the Hamiltonian function value decreases to $\mathcal{O}(10^{-6})$, but it still varies with time as shown in Figure \ref{Fig: lin_wave_red20_stand} (right).
\begin{figure}[htb]
\centering
\begin{minipage}[ht]{0.41\linewidth}
\includegraphics[width=1\textwidth]{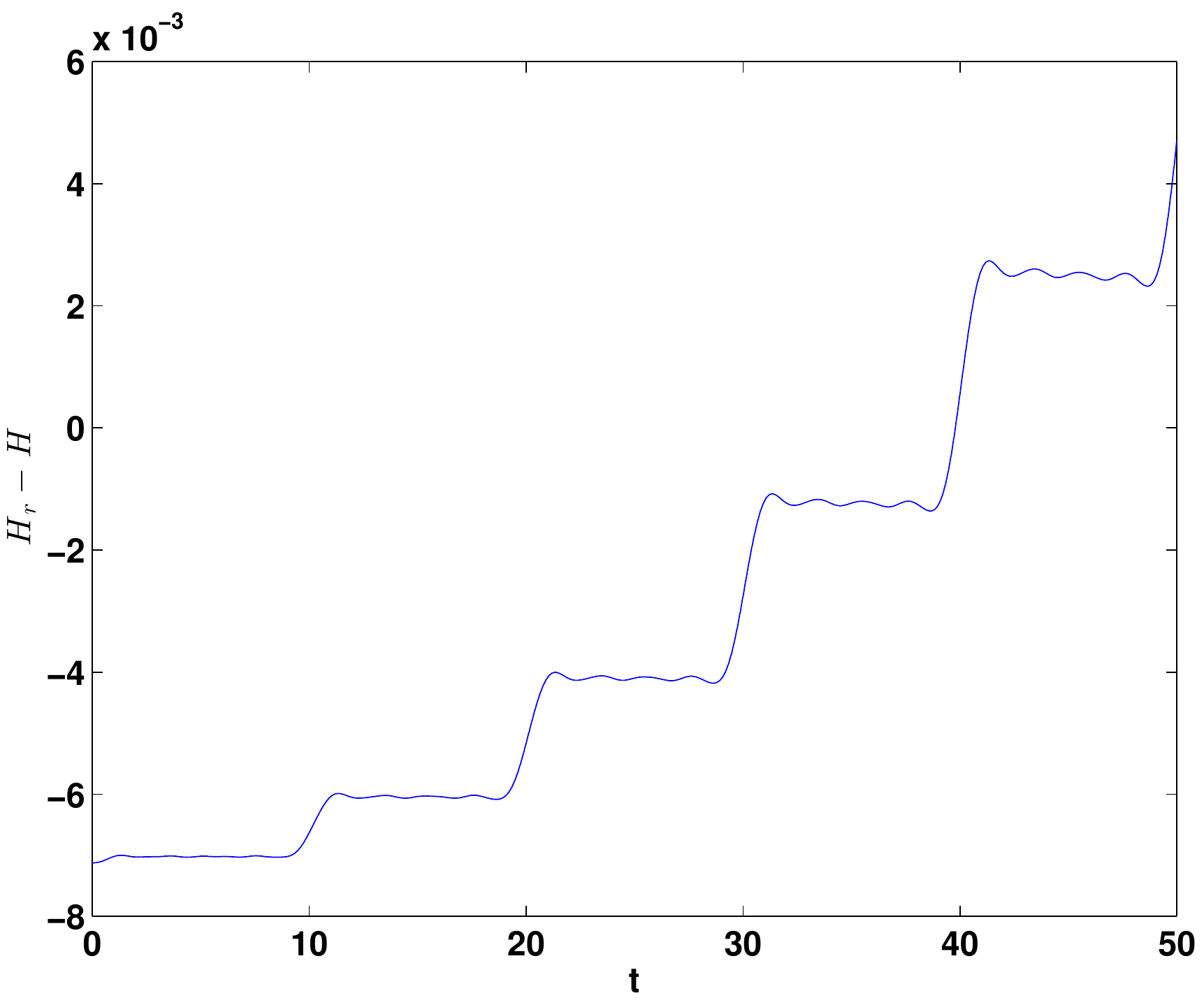}
\end{minipage}
\begin{minipage}[ht]{0.41\linewidth}
\includegraphics[width=1\textwidth]{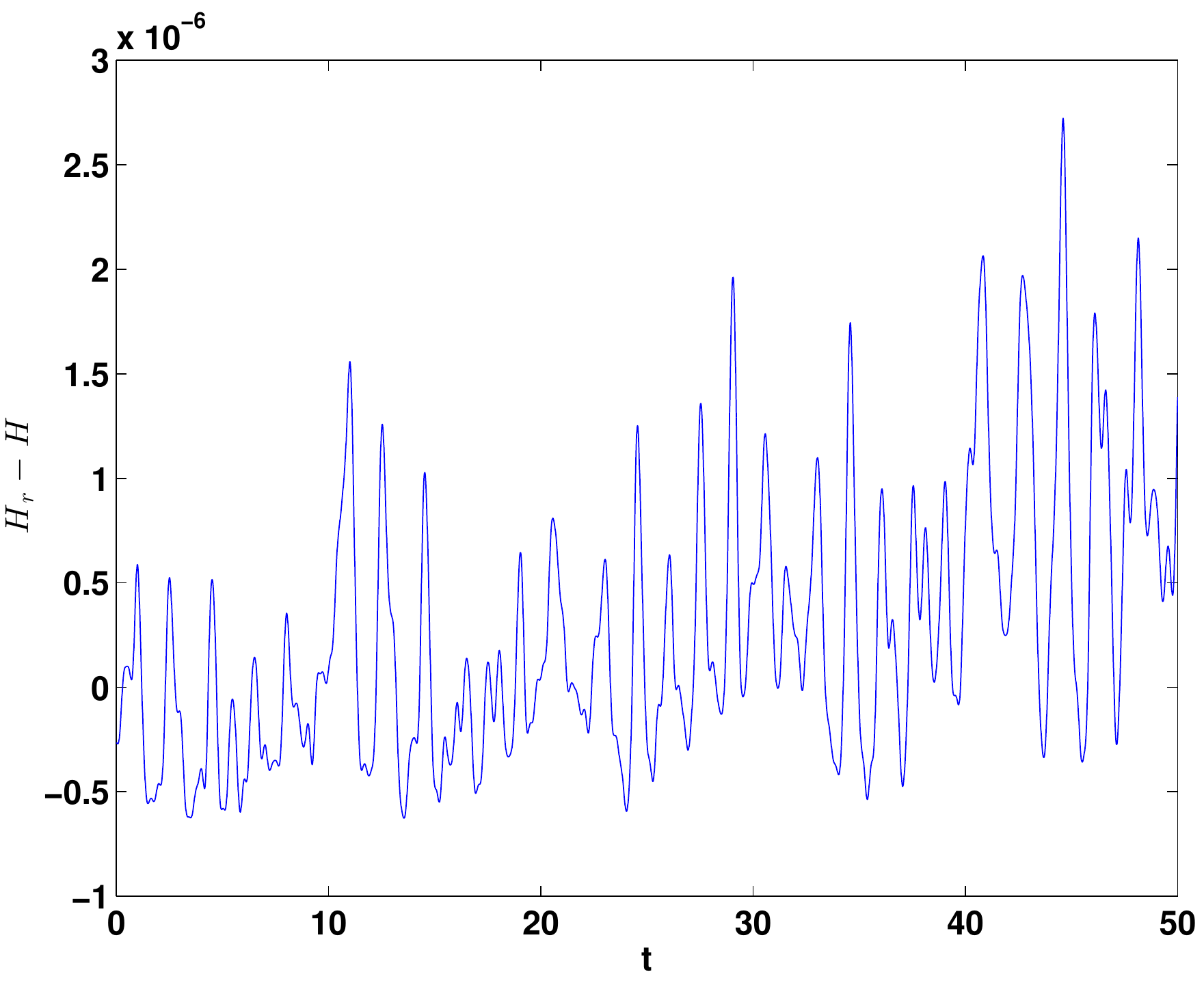}
\end{minipage}
\caption{
Time evolution of energy error in the standard POD-ROM simulations at $r=5$ (left) and $r=20$ (right).
}\label{Fig: lin_wave_red20_stand}
\end{figure}

\begin{remark}
For the same example, we didn't observe the blow up phenomena of the standard POD-G ROM solutions like what was obtained in \cite{peng2016symplectic}, even a smaller number of POD basis functions and a longer time interval were used in our test. One possible reason for the good performance is that we generated POD basis functions for $u$ and $v$ separately, not together. This ensures appropriate information was captured by each basis.
\end{remark}

\subsubsection{SP-ROMs with standard POD basis.}
\noindent \indent Based on the method proposed in Section \ref{sec: alg}, we construct a new structure-preserving ROM using the standard POD basis (SP-ROM-0):
\begin{equation}
\left[
\begin{array}{c}
\dot{\ba}\\
\dot{\bb}
\end{array}
\right] =
\left[
\begin{array}{cc}
0 & \bPhi_u^\top \bPhi_v \\
-\bPhi_v^\top \bPhi_u & 0
\end{array}
\right]
\left[
\begin{array}{c}
-\bPhi_u^\top \bA \bPhi_u \ba\\
\bb
\end{array}
\right].
\end{equation}
The coefficient matrix in the new ROM is skew-symmetric, which has the same structure as that of  the FOM.
Thus, we expect a constant Hamiltonian function approximation in the reduced-order simulation by using the AVF scheme.
The discretized equation system reads
\begin{equation}
\left[
\begin{array}{c}
\frac{\ba^{k+1}-\ba^{k}}{\Delta t}\\
\frac{\bb^{k+1}-\bb^{k}}{\Delta t}
\end{array}
\right] =
\left[
\begin{array}{cc}
0 & \bPhi_u^\top \bPhi_v \\
-\bPhi_v^\top \bPhi_u & 0
\end{array}
\right]
\left[
\begin{array}{c}
-\bPhi_u^\top \bA \bPhi_u \frac{\ba^{k+1}+\ba^k}{2}\\
\frac{\bb^{k+1}+\bb^k}{2}
\end{array}
\right]
\end{equation}
with $\ba^0 = \bPhi_u^\top \bu_0$ and $\bb^0 = \bPhi_v^\top \bv_0$.

Motivated by our error analysis, we introduce the weighted gradient of Hamiltonian function into the snapshots \eqref{eq: snap}.
We first investigate the effect of weight $\mu$ on the numerical performance of SP-ROM-0 by varying its value in the range $0\leq \mu\leq 1$ (note that $\mu$ can be any positive constant, but a larger $\mu$ would make POD basis functions capture more information from the gradient of Hamiltonian function than the state variable).

Figure \ref{Fig: lin_wave_mu} shows the trend of maximum error $\mathcal{E}_\infty$ versus $\mu$ for the cases $r=5$ and $r=20$, respectively.
It is found that, when $r=5$, the minimum $\mathcal{E}_\infty$ is achieved at $\mu= 0.0800$ with the associated error $0.2480$;
the value of $\mathcal{E}_\infty$ increases slightly to $0.2606$ at $\mu=0$.
When $r=20$, the minimum $\mathcal{E}_\infty$ is obtained at $\mu= 0.0020$ with the related error $0.0051$;
it increases slightly to $0.0058$ at $\mu=0$.

\begin{figure}[htb]
\centering
\begin{minipage}[ht]{0.38\linewidth}
\includegraphics[width=1\textwidth]{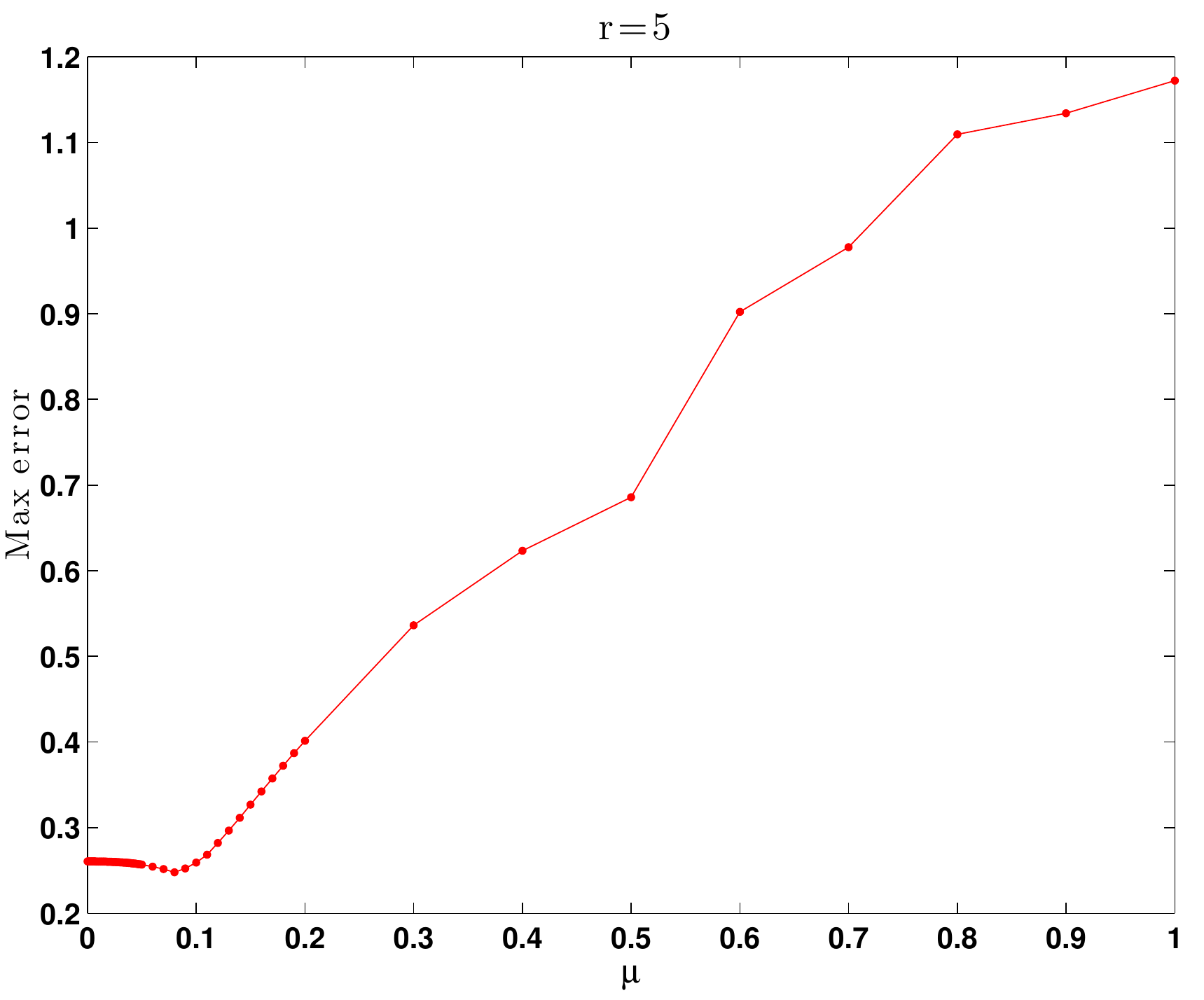}
\end{minipage}
\hspace{1cm}
\begin{minipage}[ht]{0.38\linewidth}
\includegraphics[width=1\textwidth]{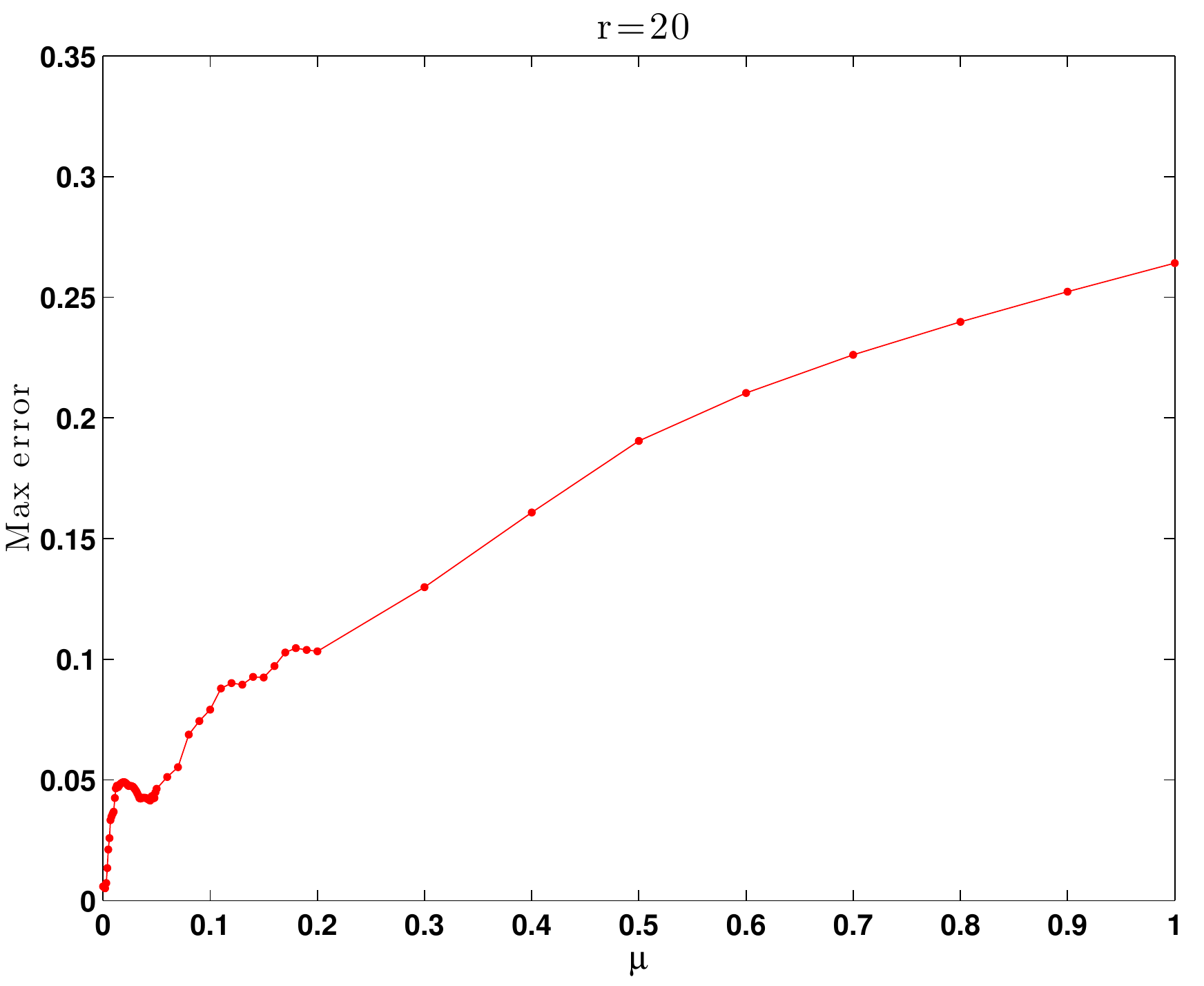}
\end{minipage}
\caption{
Maximum errors of the SP-ROM-0 approximation versus the values of $\mu$: $r=5$ (left) and $r=20$ (right).
}\label{Fig: lin_wave_mu}
\end{figure}

Based on the above observations, we conclude that: (i) when $\mu=0$, the SP-ROM-0 performs well; (ii) by choosing an optimal $\mu$, the accuracy of the SP-ROM-0 model is improved.
However,
the price one pays for finding the optimal $\mu$ is multiple runs of the ROMs, involving both offline and online processes, in order to tune this free parameter.
Since the SP-ROM yields a good approximation when $\mu=0$, we will only consider the $\mu= 0$ case in what follows.

When $\mu= 0$, that is, the snapshot set only contains the selected time samples of $\bu_h(t)$.
The Hamiltonian of the $5$-dimensional SP-ROM-0, $H_r(t)$, is a constant as shown in Figure \ref{Fig: lin_wave_red5} (left).
However, there exists a discrepancy $H_r(t)-H(t)= -7.1245\times 10^{-3}$.
As the dimension is increased to $r=20$, the reduced-order simulation is improved as expected.
The global error decreases from $0.2606$ to $0.0058$, and the Hamiltonian is close to that of the FOM.
Indeed, the energy discrepancy shrinks to be $H_r(t)-H(t)= -2.6563\times 10^{-7}$ as shown in Figure \ref{Fig: lin_wave_red5} (right).
\begin{figure}[htb]
\centering
\begin{minipage}[ht]{0.41\linewidth}
\includegraphics[width=1\textwidth]{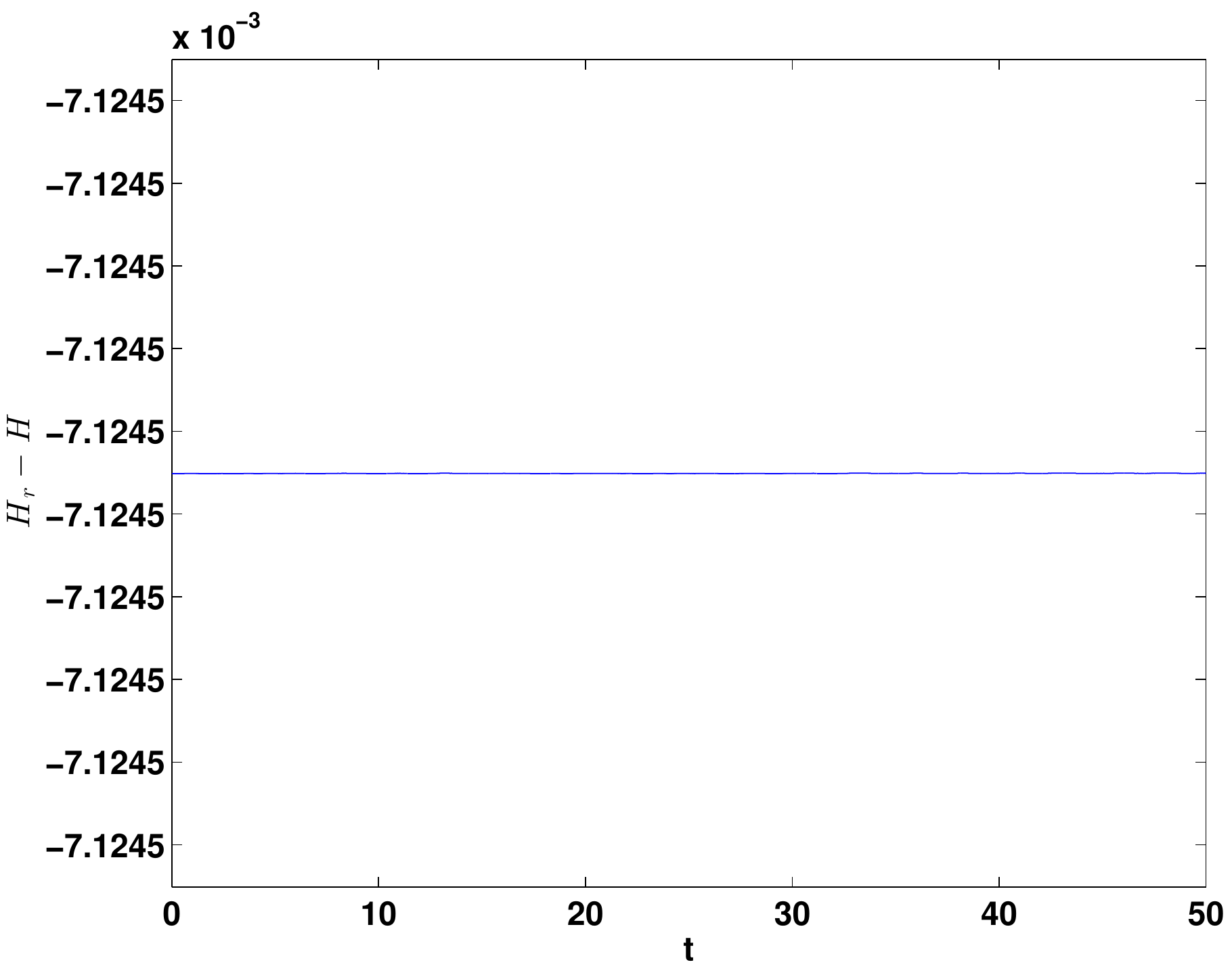}
\end{minipage}
\hspace{1cm}
\begin{minipage}[ht]{0.41\linewidth}
\includegraphics[width=1\textwidth]{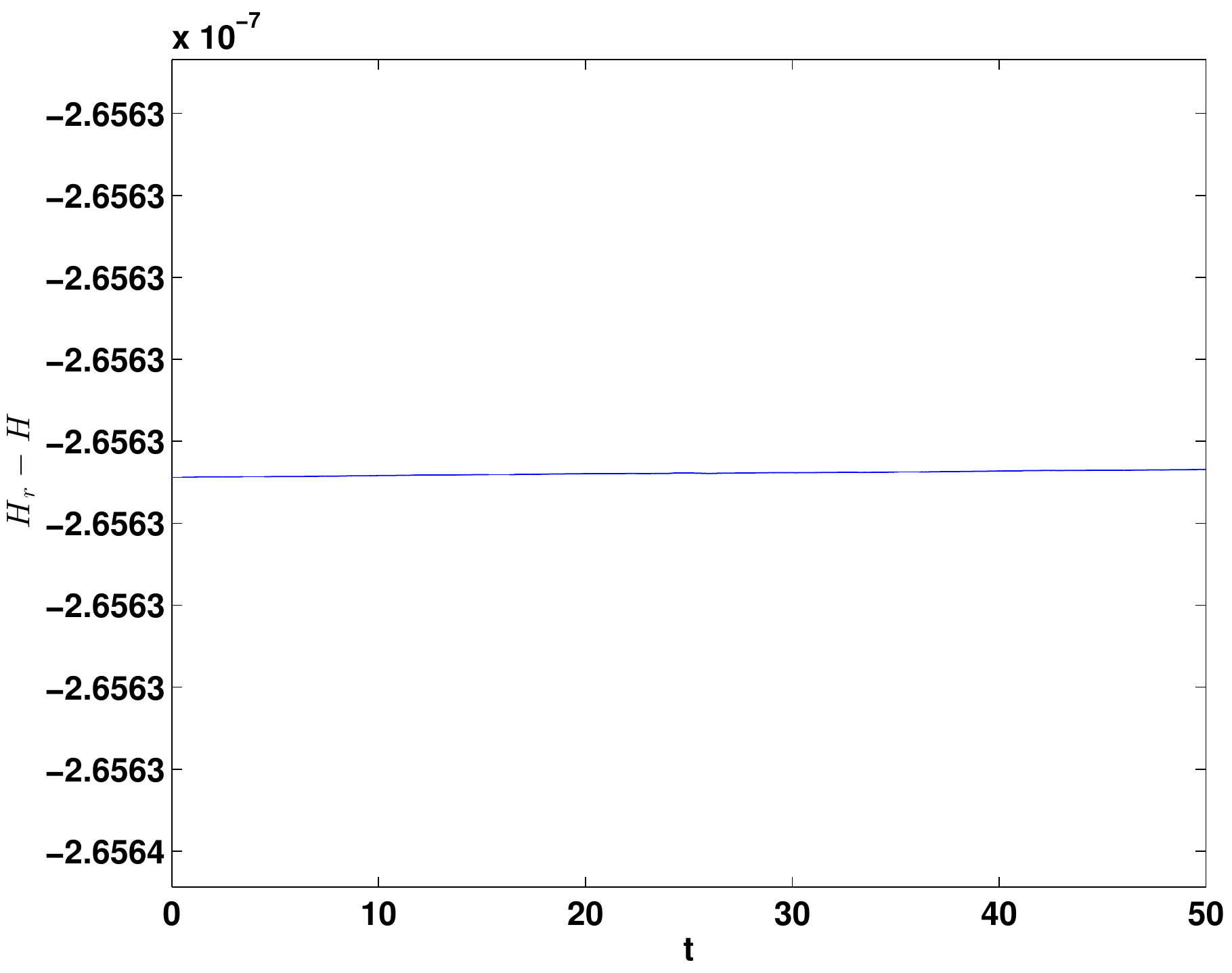}
\end{minipage}
\caption{
Time evolution of energy error of SP-ROM-0 at $\mu=0$: $r=5$ (left) and $r=20$ (right).
}\label{Fig: lin_wave_red5}
\end{figure}

In practice, it is more useful to improve the performance of low-dimensional ROMs.
In the next two subsections, we will focus on the 5-dimensional structure preserving ROMs $(r= 5)$, $\mu=0$ in the snapshots, and compare two ways for improving the Hamiltonian approximation.

\subsubsection{SP-ROMs with corrected Hamiltonian}

 \paragraph{Approach I. SP-ROMs with enriched POD basis}
We introduce into the 5-dimensional POD basis set an extra basis function generated from the residual of initial data, that is, $\widetilde{\bPhi}_u= [\bPhi_u, \bpsi_u]$ and $\widetilde{\bPhi}_v= [\bPhi_v, \bpsi_v]$.
The discrete structure preserving ROM with the enriched POD basis (SP-ROM-1) is the following:
\begin{equation}
\left[
\begin{array}{c}
\frac{\ba^{k+1}-\ba^{k}}{\Delta t}\\
\frac{\bb^{k+1}-\bb^{k}}{\Delta t}
\end{array}
\right] =
\left[
\begin{array}{cc}
0 & \widetilde{\bPhi}_u^\top \widetilde{\bPhi}_v \\
-\widetilde{\bPhi}_v^\top \widetilde{\bPhi}_u & 0
\end{array}
\right]
\left[
\begin{array}{c}
-\widetilde{\bPhi}_u^\top \bA \widetilde{\bPhi}_u \frac{\ba^{k+1}+\ba^k}{2}\\
\frac{\bb^{k+1}+\bb^k}{2}
\end{array}
\right]
\end{equation}
with $\ba^0 = \widetilde{\bPhi}_u^\top \bu_0$ and $\bb^0 = \widetilde{\bPhi}_v^\top \bv_0$.
Since the SP-ROM \eqref{eq: ham_rom2}-\eqref{eq: Sr} keeps the Hamiltonian function invariant, it is expected that the energy approximation in the SP-ROM-1 simulations becomes more accurate after introducing the extra basis function.

Consider the snapshots at $\mu= 0$, that is, the snapshot set \eqref{eq: snap} only contains the selected time samples of $\bu_h(t)$.
It is found that there is no any visible discrepancy in the energy approximation between the ROM and FOM.
As shown in Figure \ref{Fig: lin_wave_icenrich5} (left), the magnitude of the discrepancy is $\mathcal{O}(10^{-14})$.
However, the maximum error increases to $\mathcal{E}_\infty= 0.4138$, which is more than 1.5 times of that in the 5-dimensional SP-ROM-0.
\begin{figure}[htb]
\centering
\begin{minipage}[ht]{0.41\linewidth}
\includegraphics[width=1\textwidth]{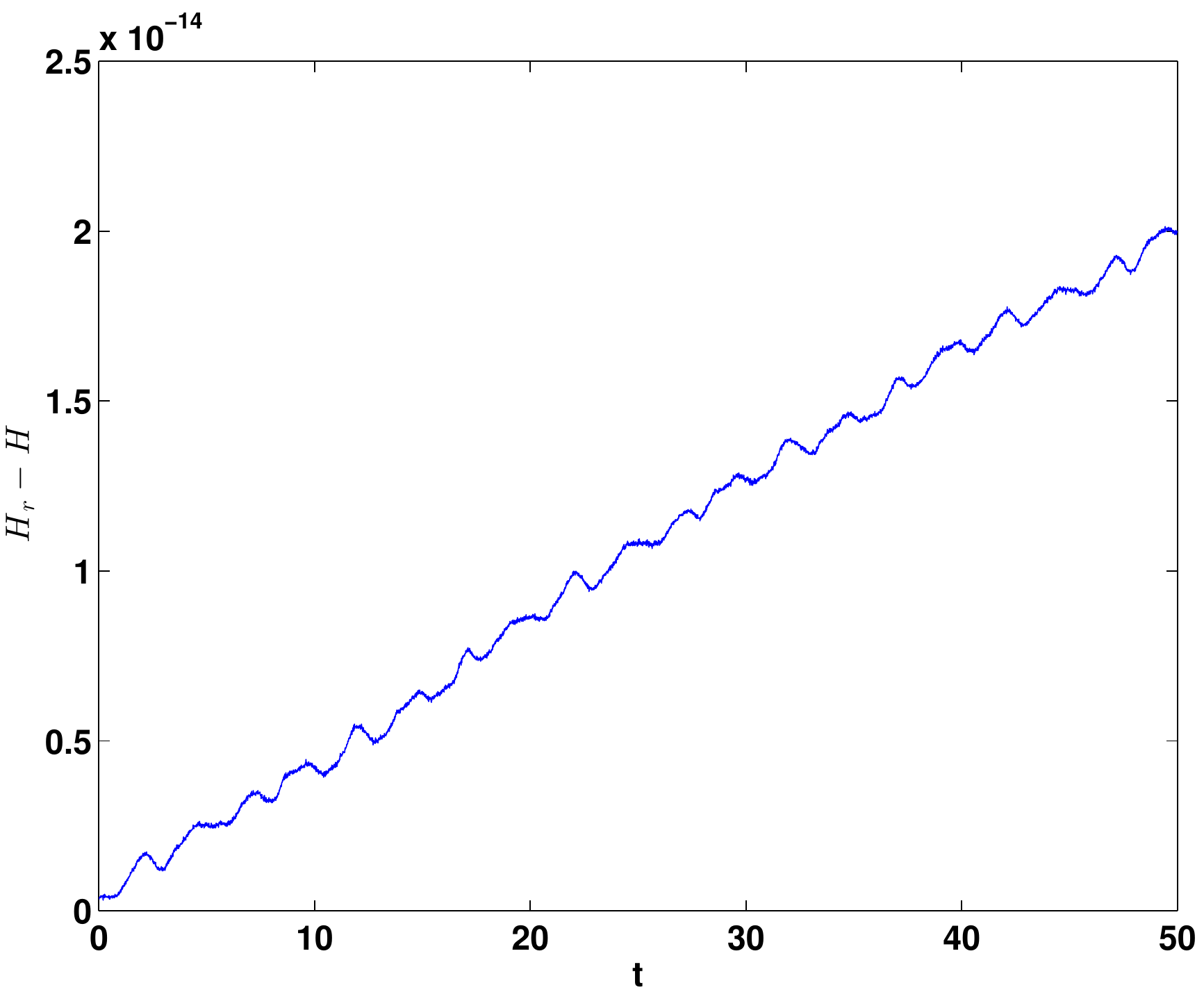}
\end{minipage}
\hspace{1cm}
\begin{minipage}[ht]{0.41\linewidth}
\includegraphics[width=1\textwidth]{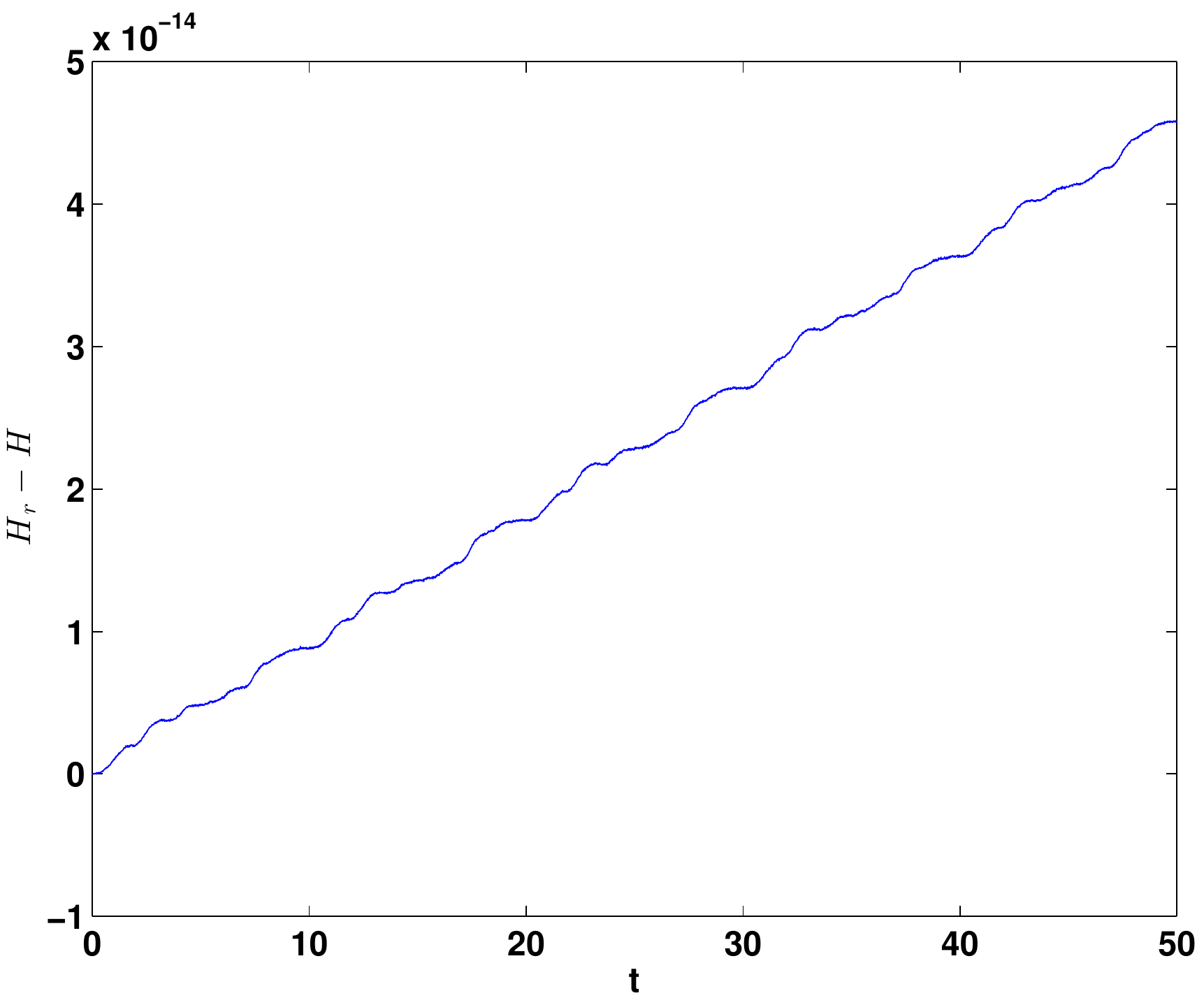}
\end{minipage}
\caption{
Time evolution of energy error in the 6-dimensional SP-ROM-1 (left) and 5-dimensional SP-ROM-2 (right) when $\mu=0$. Note that the magnitude of $H_r(t)-H(t)$ is $\mathcal{O}(10^{-14})$, thus the energy is accurately captured in both models.
}\label{Fig: lin_wave_icenrich5}
\end{figure}

\paragraph{Approach II. SP-ROMs with POD basis of shifted snapshots}  In this case, the POD basis is generated from shifted snapshots \eqref{eq: snap_shift} with the choice of $\mu= 0$.
With the use of the structure-preserving ROM \eqref{eq: ham_rom3} and \eqref{eq: Sr}, the SP-ROM-2 model reads
\begin{equation}
\left[
\begin{array}{c}
\frac{\ba^{k+1}-\ba^{k}}{\Delta t}\\
\frac{\bb^{k+1}-\bb^{k}}{\Delta t}
\end{array}
\right] =
\left[
\begin{array}{cc}
0 & \bPhi_u^\top \bPhi_v \\
-\bPhi_v^\top \bPhi_u & 0
\end{array}
\right]
\left(
\left[
\begin{array}{c}
-\bPhi_u^\top \bA \bPhi_u \frac{\ba^{k+1}+\ba^k}{2}\\
\frac{\bb^{k+1}+\bb^k}{2}
\end{array}
\right]
+
\left[
\begin{array}{c}
-\bPhi_u^\top \bA \bu_0\\
\bPhi_v^\top\bv_0
\end{array}
\right]
\right)
\end{equation}
with $\ba^0 = {\bf 0}$ and $\bb^0 = {\bf 0}$.

The time evolution of the Hamiltonian function error, $H_r(t)-H(t)$, is shown in Figure \ref{Fig: lin_wave_icenrich5} (right).
It is observed that the magnitude of the discrepancy is $\mathcal{O}(10^{-14})$, thus the Hamiltonian function approximation in the SP-ROM-2 is more accurate.
Furthermore, the maximum error $\mathcal{E}_\infty= 0.1526$, which is about 59\% of the 5-dimensional SP-ROM-0 approximation error.

\subsection{Korteweg-de Vries (KdV) equation\label{sec: kdv}}
\noindent \indent The KdV equation
\begin{equation*}
u_t = \alpha uu_x + \rho u_x + \gamma u_{xxx},
\end{equation*}
defined in the spatial temporal domain $[-L, L]\times [0, T]$, has a bi-Hamiltonian form (see, e.g., \cite{Karasozen2013energy}).
Here, we consider the first Hamiltonian formulation
\begin{equation}
u_t = \mathcal{S} \frac{\delta \mathcal{H}}{\delta u},
\end{equation}
where $\mathcal{S}= \partial_x$ denotes the first-order derivative with respect to space and the Hamiltonian function, which is the system energy,
$$\mathcal{H}=\int_0^L \left(\frac{\alpha}{6}u^3+\frac{\rho}{2}u^2-\frac{\nu}{2}u_x^2 \right) \, dx.$$
The other Hamiltonian formulation can be treated in the same manner.

Consider a problem associated with the periodic boundary conditions $u(-L, t)= u(L, t)$ for $t\in [0, T]$ and initial condition $u(x, 0)=u_0(x)$. In the full-order simulation, the domain $[-L, L]$ is divided into $n$ uniform subdomains with the interior grid points $x_i= i\Delta x$ for $i= 1, \ldots, n$ and $\Delta x= \frac{2L}{n}$.
Let $\bA$ and $\bB$ be the matrices associated to the discretization of the skew adjoint operator $\mathcal{S}$ and the second-order derivative by central differences, respectively, i.e.,
\begin{equation}
\bA= \frac{1}{2\Delta x}
\left( \begin{array}{cccccc}
0 & 1 & 0  & 0 & \cdots & -1 \\
-1 & 0 & 1 & 0 & \cdots & 0 \\
   &     & \ddots & \ddots & \ddots &  \\
 0   &    \cdots&0          &  -1 & 0  & 1 \\
 1 &    \cdots      &0     &  0  &  -1 & 0
\end{array} \right),
\qquad
\bB= \frac{1}{\Delta x^2}
\left( \begin{array}{cccccc}
-2 & 1 & 0  & 0 & \cdots & 1 \\
1 & -2 & 1 & 0 & \cdots & 0 \\
   &     & \ddots & \ddots & \ddots &  \\
 0   &    \cdots&0          &  1 & -2  & 1 \\
 1 &    \cdots      &0     &  0  &  1 & -2
\end{array} \right)
\end{equation}
and $\bu= (u_1, \ldots, u_n)^\top$,
the semi-discrete KdV equation can be written in a vector form
\begin{eqnarray}
\frac{d\bu}{dt} &=& \bA \nabla_{\bu}H(\bu) \nonumber \\
		      &=& \bA\left(\frac{\alpha}{2}\bu^2+\rho\bu+\nu\bB\bu\right),
\label{eq: kdv_fom1}
\end{eqnarray}
where $H(\bu)= \sum_{j=1}^n\left[\frac{\alpha}{6}u_j^3 + \frac{\rho}{2}u_j^2 - \frac{\nu}{2}(\delta_x^+ u_j)^2\right]$ and $\delta_x^+ u_j$ is the forward finite differencing. The discrete energy $H\Delta x$ approximates to $\mathcal{H}$ as $\Delta x$ goes to zero. Since $\bA$ is skew-symmetric, this dynamical system conserves the discrete energy $H\Delta x$.

In the rest of this subsection, we consider the test example, in which parameters $\alpha= -6$, $\rho= 0$, $\nu= -1$, $L= 20$ and $T=20$.
The initial condition $u_0(x)= {\sech}^2\left(\frac{x}{\sqrt{2}}\right)$.
Mesh sizes are chosen as $\Delta x= \Delta t= 0.02$ in the FOM simulation.
For the time integration, we use the method of AVF.
At time $t_k$, the solution $\bu_h^k$ satisfies,
\begin{equation}
\frac{\bu_h^{k+1}-\bu_h^k}{\Delta t} = \bA\left[\frac{\alpha}{6}\left( (\bu_h^k)^2+\bu_h^k\bu_h^{k+1}+(\bu_h^{k+1})^2 \right) + \rho \bu_h^{k+\frac{1}{2}} + \nu\bB\bu_h^{k+\frac{1}{2}}\right],
\label{eq: kdv_fom}
\end{equation}
where $\bu_h^{k+\frac{1}{2}}= \left(\bu_h^k+\bu_h^{k+1}\right)/2$
and the initial data $\bu_h^0$ has the $i$-th component to be $u_0(x_i)$ for $1\leq i\leq n$.
Due to the nonlinearity of the model, iterative methods such as the Picard's method is used when solving the system.
The evolution of the full-order simulation $u_h(x, t)$ and energy $\mathcal{H}(t)\approx -1.1317$ is shown in Figure \ref{Fig: kdv_fom}.
\begin{figure}[htb]
\centering
\begin{minipage}[ht]{0.38\linewidth}
\includegraphics[width=1\textwidth]{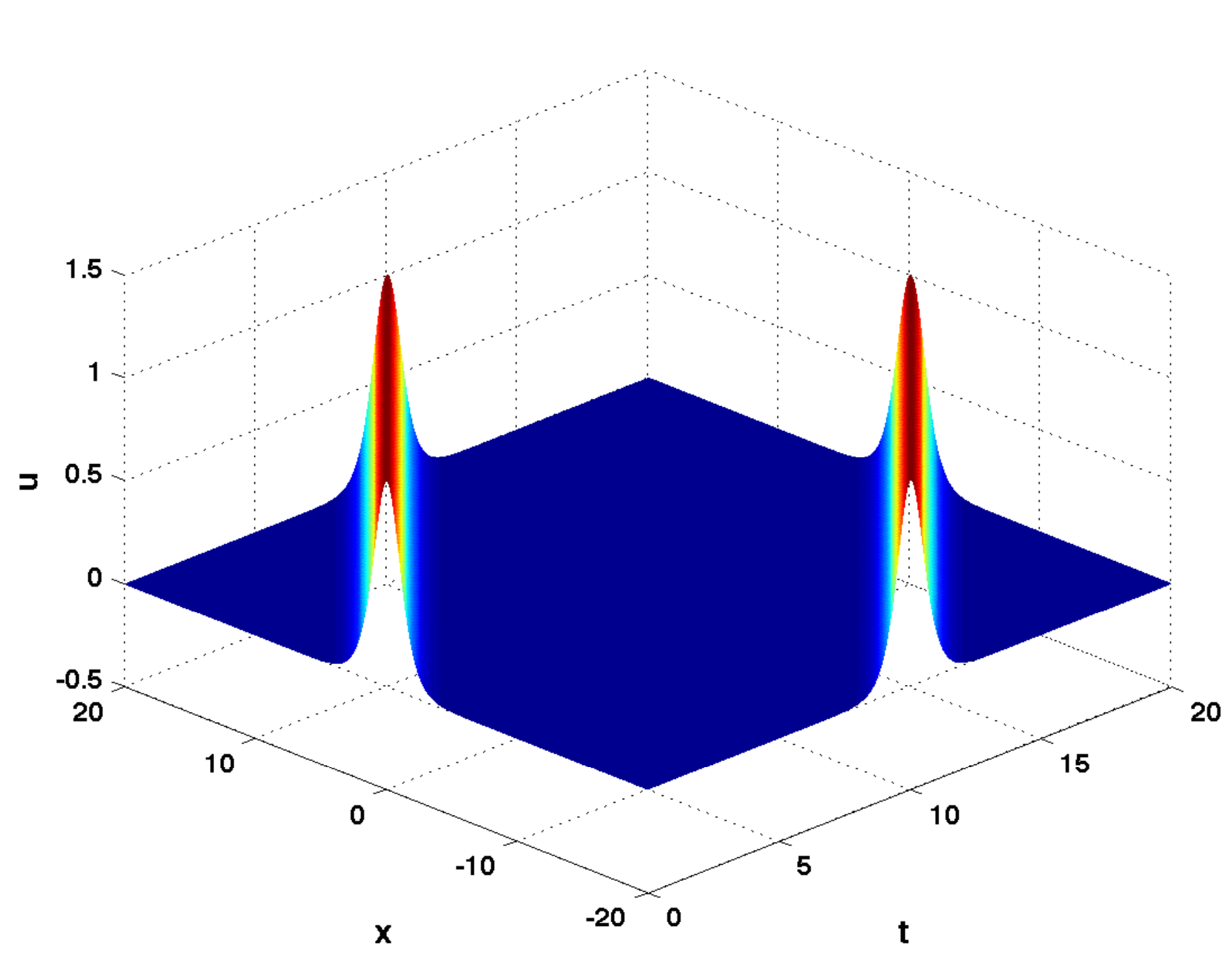}
\end{minipage}
\hspace{1cm}
\begin{minipage}[ht]{0.38\linewidth}
\includegraphics[width=1\textwidth]{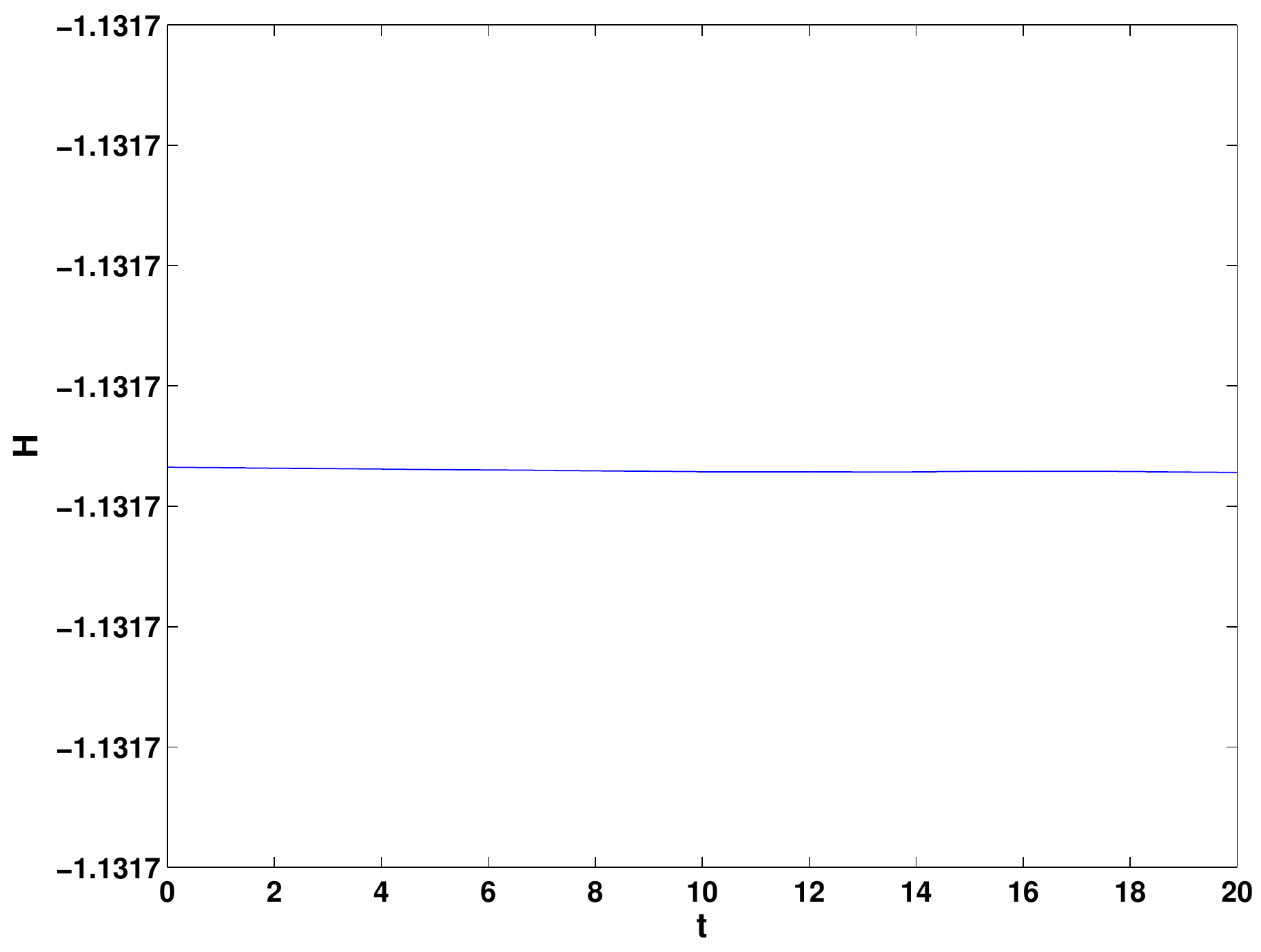}
\end{minipage}
\caption{The full-order state solution $u$ (left) and energy $\mathcal{H}(t)$ (right).
}\label{Fig: kdv_fom}
\end{figure}

Next, we investigate the numerical behavior of ROMs.
Since the exact solution is unknown, we regard the full-order results as a benchmark and compare the accuracy of standard POD-G ROMs and the proposed structure-preserving ROMs by measuring the error $\mathcal{E}_\infty= \max\limits_{k>0} \max\limits_{0\leq i\leq n} |(\bu_h^k)_i-(\bu_r^k)_i|$
and the energy value in the reduced-order simulations, $\mathcal{H}_r(t)$.

\subsubsection{Standard POD-G ROMs}
\noindent \indent Considering the snapshots collected from the full-order simulation every $5$ time steps, we generate
the $r$-dimensional POD basis matrix $\bPhi$. Replacing $\bu$ by its reduced-order approximation $\bu_r=\bPhi\ba$ in \eqref{eq: kdv_fom1}, multiplying $\bPhi^\top$ on both sides of the equation and using the fact that $\bPhi^\top\bPhi= I_r$, we have
\begin{equation}
\frac{d\ba}{dt} = \bPhi^\top \bA \nabla_{\bu}H(\bPhi \ba).
\label{eq: kdv_rom1}
\end{equation}
Using the same AVF scheme, the discrete POD-G ROM reads
\begin{equation}
\frac{\ba^{k+1}-\ba^k}{\Delta t} =
\bPhi^\top\bA\left[\frac{\alpha}{6}\left( (\bPhi\ba^k)^2
+(\bPhi\ba^k)(\bPhi\ba^{k+1})
+(\bPhi\ba^{k+1})^2 \right)
+ \rho \bPhi \frac{\ba^k+\ba^{k+1}}{2}
+ \nu\bB\bPhi \frac{\ba^k+\ba^{k+1}}{2}\right]
\label{eq: kdv_rom_e1}
\end{equation}
with the initial condition $\ba^0= \bPhi^\top \bu^0$.
This model is also nonlinear, thus the Picard's iteration is used in simulating the system.
Since the nonlinear is quadratic, tensor manipulation can be used for efficiently evaluating the nonlinear term (see, e.g., in \cite{wang2011closure}).

When $r= 40$, the maximum error $\mathcal{E}_\infty= 2.9637\times 10^{-2}$.
The time evolution of energy is shown in Figure \ref{Fig: kdv_rom_r40} (left), which indicates the energy is not exactly conserved.
It is because the coefficient matrix of the ROM is not skew-symmetric, thus the Hamiltonian structure is damaged.
\begin{figure}[htb]
\centering
\begin{minipage}[ht]{0.41\linewidth}
\includegraphics[width=1\textwidth]{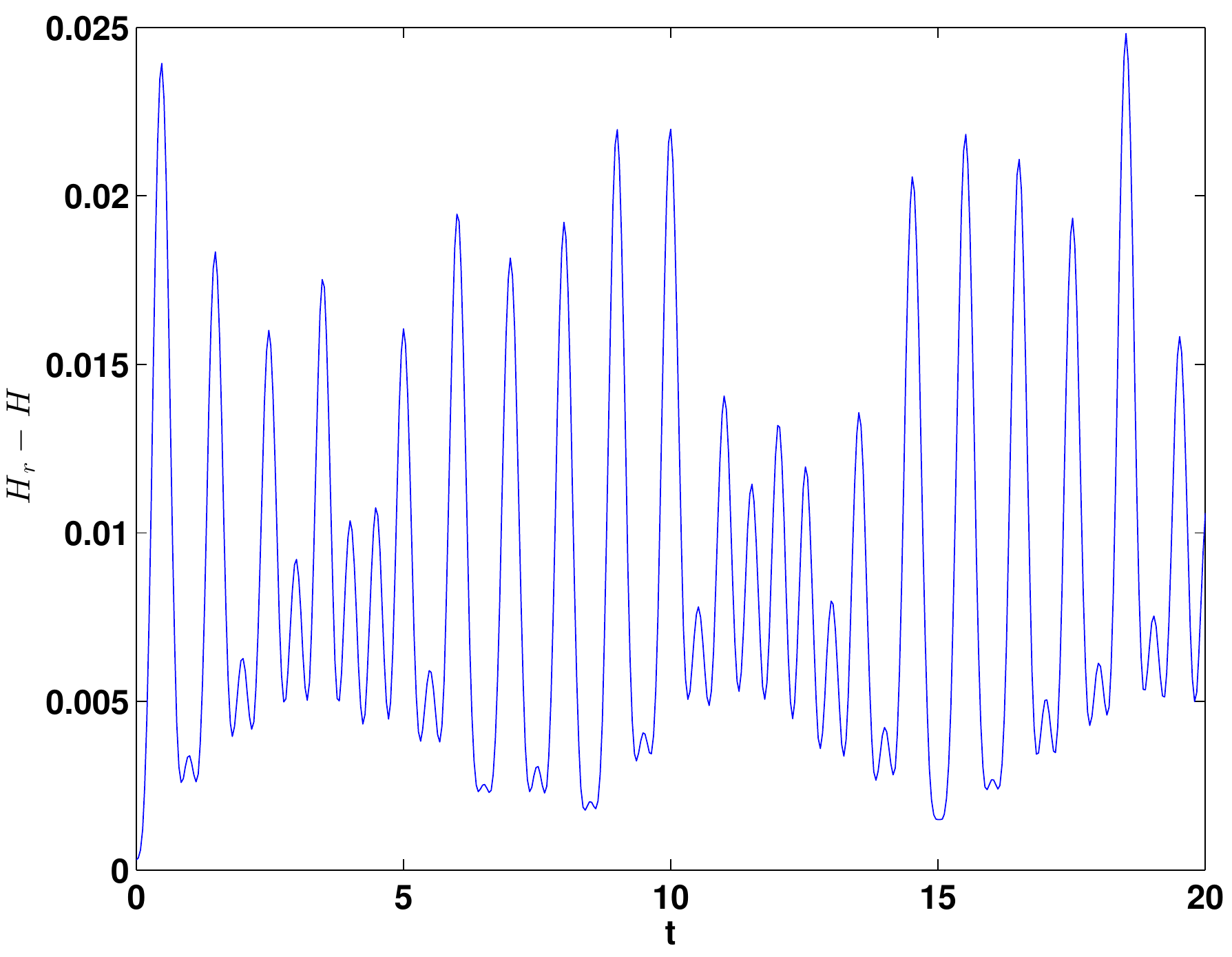}
\end{minipage}
\hspace{1cm}
\begin{minipage}[ht]{0.41\linewidth}
\includegraphics[width=1\textwidth]{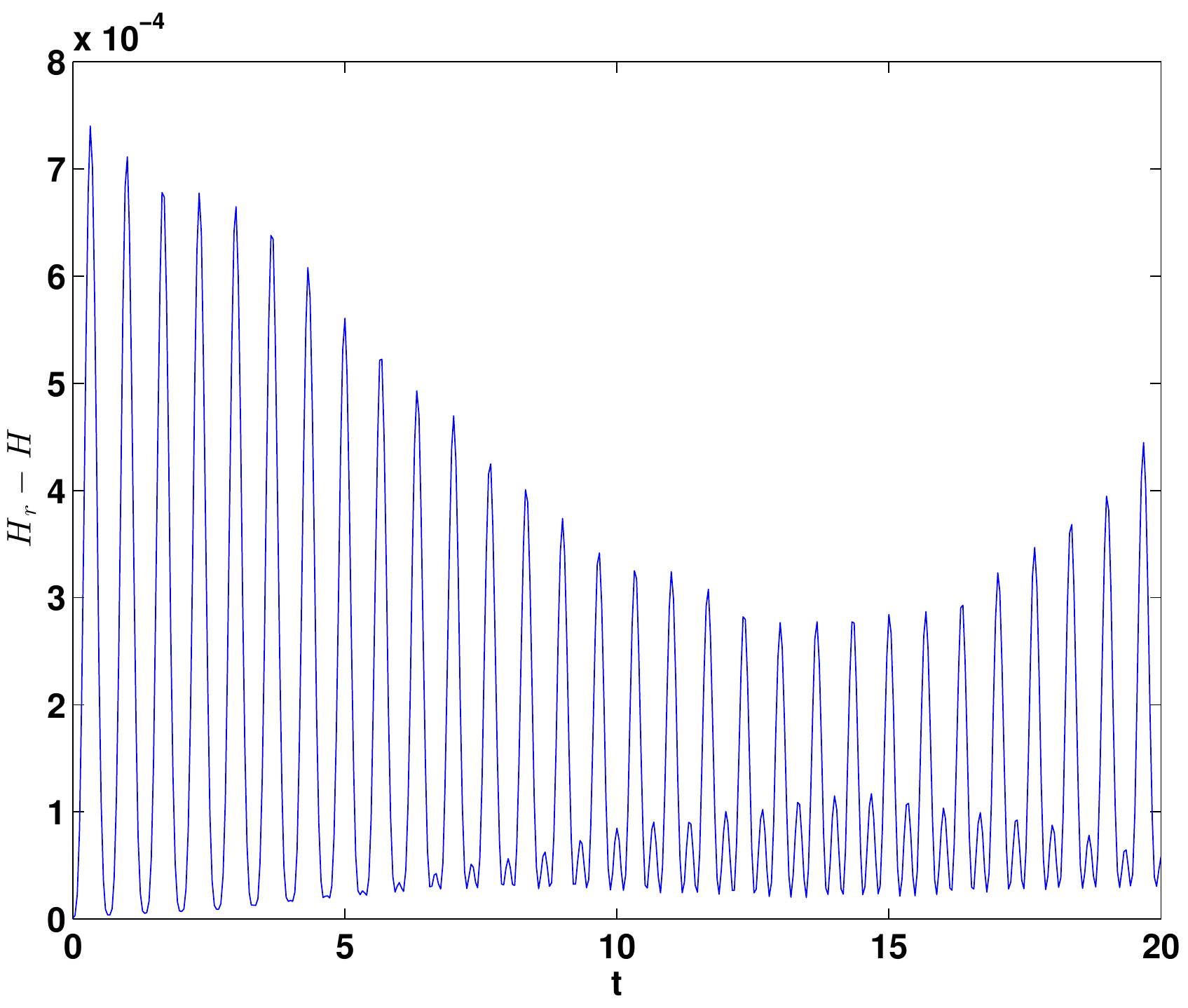}
\end{minipage}
\caption{Time evolution of energy error in the standard reduced-order approximation: $r=40$ (left) and $r=60$ (right).
}\label{Fig: kdv_rom_r40}
\end{figure}

When the dimension is increased to $r=60$, the maximum error is reduced to $\mathcal{E}_\infty= 2.4476\times 10^{-3}$, and the energy discrepancy decreases to $\mathcal{O}(10^{-4})$, but the value of the energy approximation still varies with time as shown in Figure \ref{Fig: kdv_rom_r40} (right).

In the rest of this test, we will focus on the low dimensional case and aims to improve the accuracy of the SP-ROMs with $r= 40$.

\subsubsection{SP-ROMs with standard POD basis}
\noindent \indent The proposed structure-preserving ROMs possess a skew-symmetric coefficient matrix $A_r= \bPhi^\top \bA \bPhi$, and the dynamical system reads
\begin{equation}
\frac{d\ba}{dt} = \bA_r \nabla_{\ba}H(\bPhi \ba).
\label{eq: kdv_rom2}
\end{equation}
The structure-preserving ROM with standard POD basis (SP-ROM-0) after the discretization reads
\begin{equation}
\frac{\ba^{k+1}-\ba^k}{\Delta t} =
\bA_r\bPhi^\top\left[\frac{\alpha}{6}\left( (\bPhi\ba^k)^2
+(\bPhi\ba^k)(\bPhi\ba^{k+1})
+(\bPhi\ba^{k+1})^2 \right)
+ \rho \bPhi \frac{\ba^k+\ba^{k+1}}{2}
+ \nu\bB\bPhi \frac{\ba^k+\ba^{k+1}}{2}\right]
\label{eq: kdv_rom3}
\end{equation}
with $\ba^0 = \bPhi^\top \bu^0$.

We first study the impact of $\mu$, which is the weight of $\nabla_u H(u)$ in the snapshots, on the numerical performance of structure-preserving ROMs \eqref{eq: kdv_rom3} by varying the value of $\mu$.
As a criterion, the error $\mathcal{E}_\infty$ is evaluated for accuracy.

Figure \ref{Fig: lin_kdv_mu} shows the trend of $\mathcal{E}_\infty$ versus $\mu$ for $r= 40$ and $r= 60$.
When $r= 40$, it is found that the minimum error is achieved at $\mu= 0$ and $\mathcal{E}_\infty= 0.0564$.
But when the dimension of ROM increases to $r= 60$, the minimum error is achieved at $\mu= 1$ and $\mathcal{E}_\infty= 7.3064\times 10^{-4}$; when $\mu=0$, the error $\mathcal{E}_\infty= 7.3882\times 10^{-4}$.
\begin{figure}[htb]
\centering
\begin{minipage}[ht]{0.38\linewidth}
\includegraphics[width=1\textwidth]{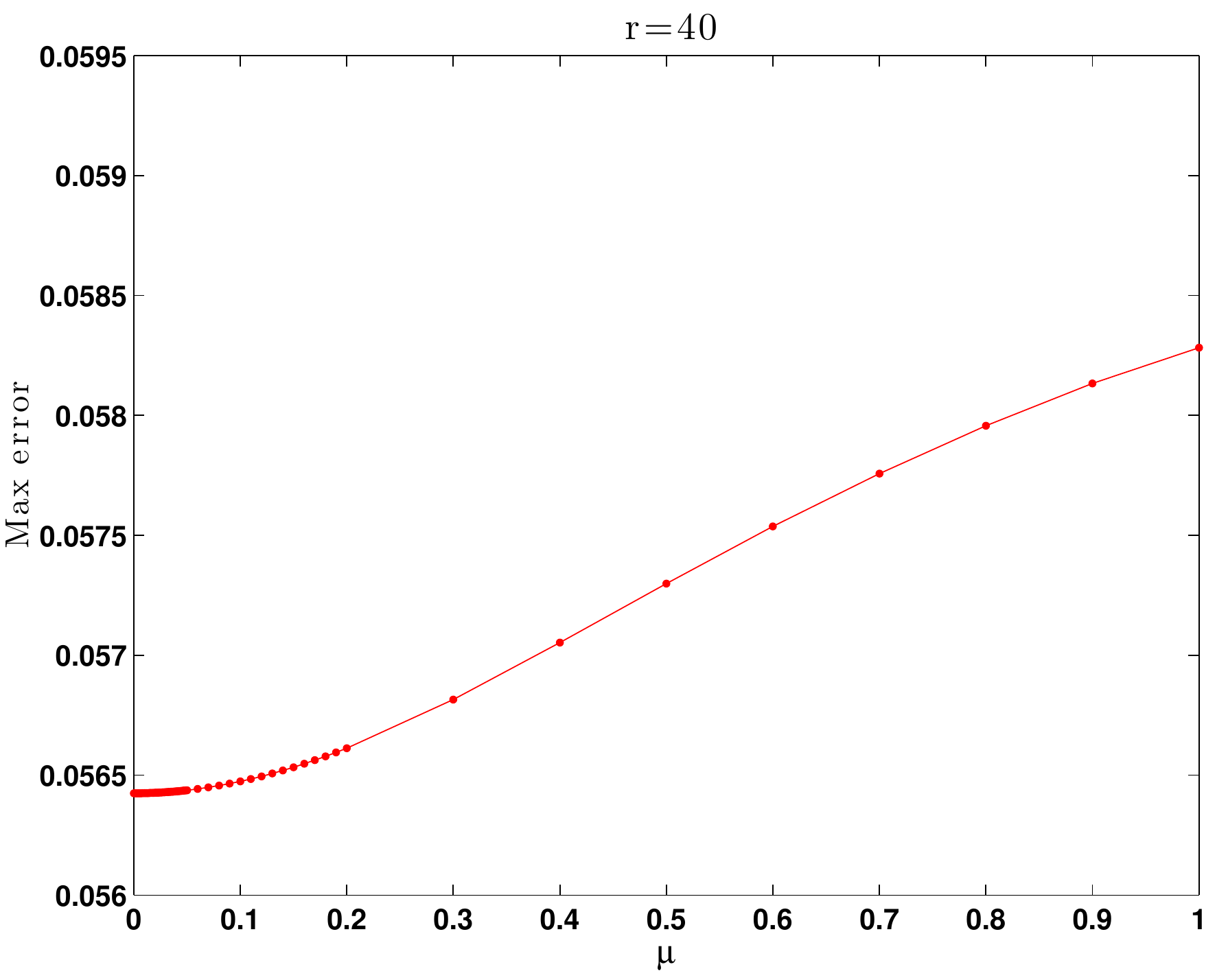}
\end{minipage}
\hspace{1cm}
\begin{minipage}[ht]{0.38\linewidth}
\includegraphics[width=1\textwidth]{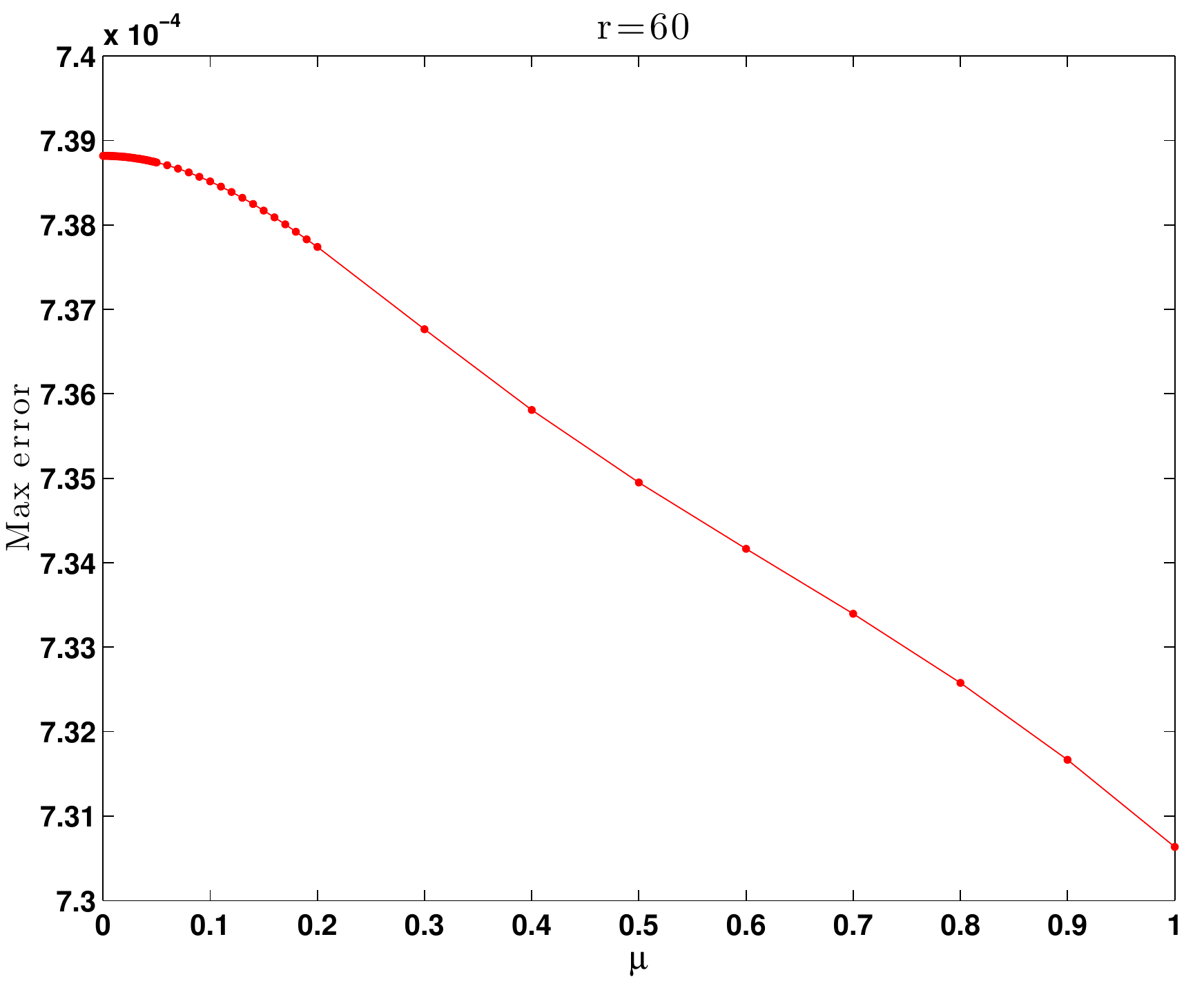}
\end{minipage}
\caption{
Maximum errors of the SP-ROM-0 simulation results versus the values of $\mu$: $r=40$ (left) and $r=60$ (right).
}\label{Fig: lin_kdv_mu}
\end{figure}

It is observed that overall, the proposed structure-preserving ROM achieves better approximation as $r$ increases;
it produces a good numerical solution when $\mu=0$ in general, and the results can be further improved by choosing $\mu$ optimally.
However, to avoid the price paid for tuning the free parameter, in the rest of this example, we will focus on the case $\mu= 0$.

\begin{figure}[htb]
\centering
\begin{minipage}[ht]{0.41\linewidth}
\includegraphics[width=1\textwidth]{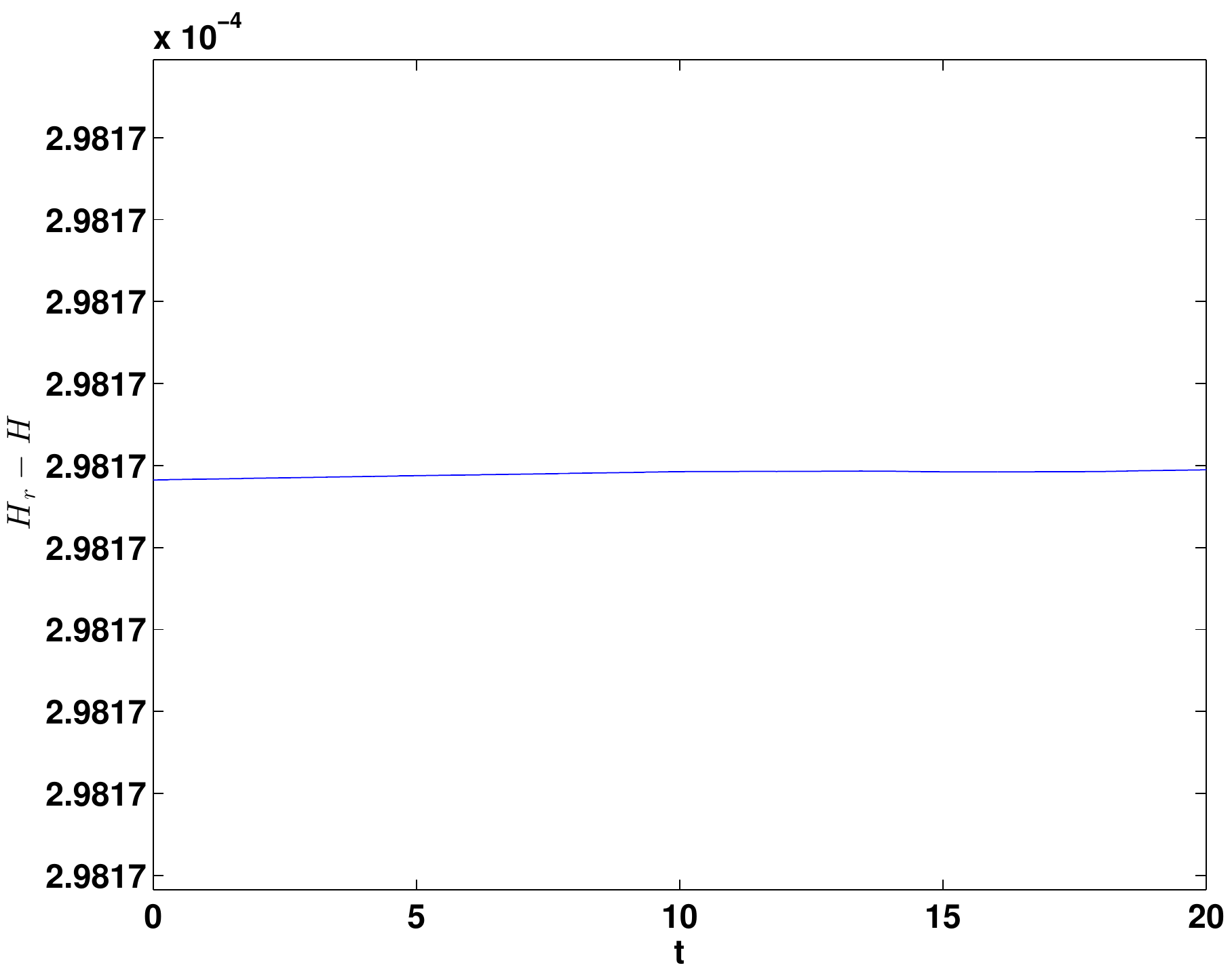}
\end{minipage}
\hspace{1cm}
\begin{minipage}[ht]{0.41\linewidth}
\includegraphics[width=1\textwidth]{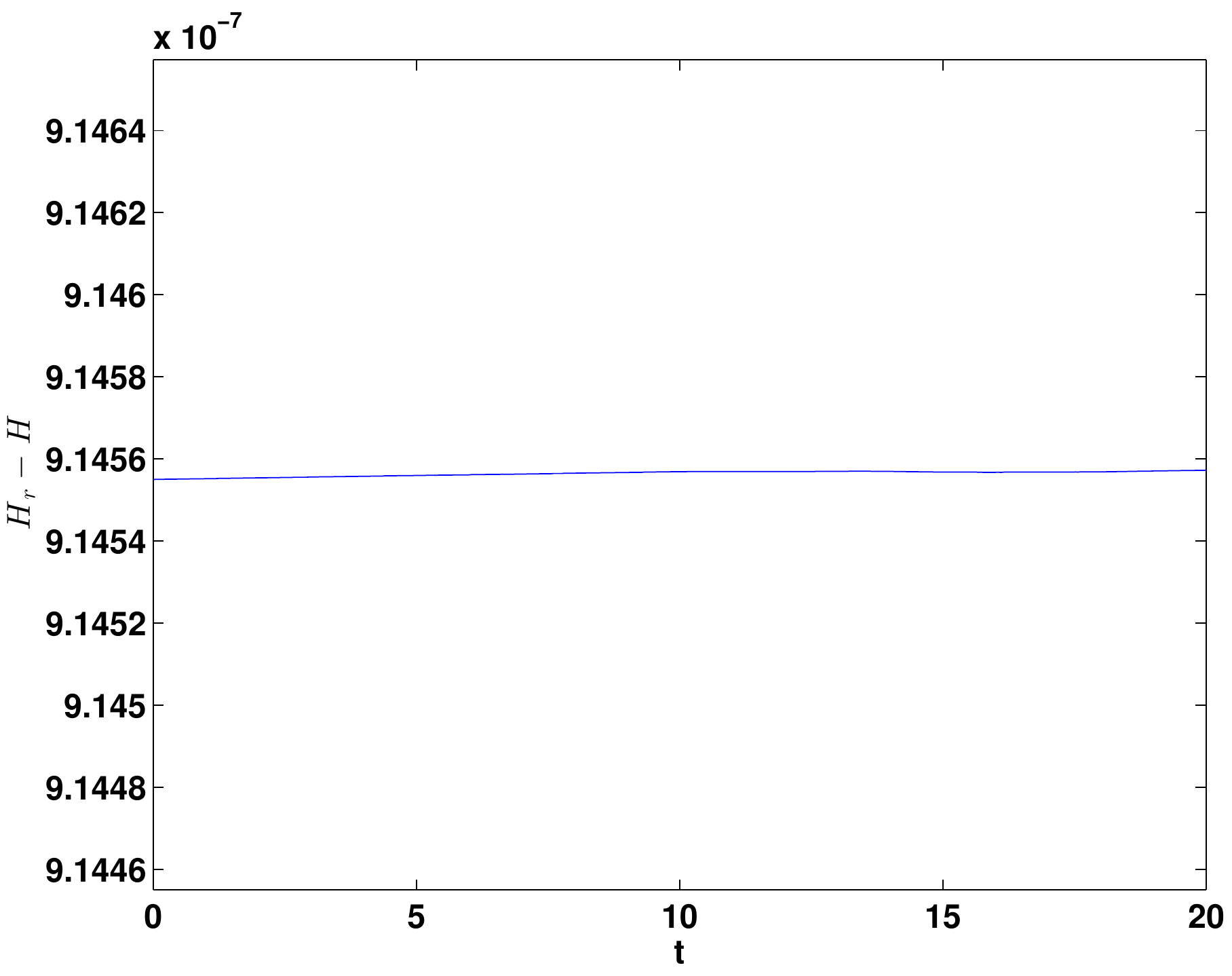}
\end{minipage}
\caption{Time evolution of the Hamiltonian function approximation error $\mathcal{H}_r(t)-\mathcal{H}(t)$ at $\mu= 0$: $r= 40$ (left) and $r= 60$ (right).
Note that the error magnitudes do not change with time during the simulations.
}\label{Fig: kdv_sprom0_r40}
\end{figure}

The time evolution of energy approximation errors in SP-ROM-0 with $\mu= 0$ are shown in Figure \ref{Fig: kdv_sprom0_r40} for $r= 40$ and $r= 60$, respectively.
It is seen that, when the dimension increases from 40 to 60, the magnitude of the energy discrepancy reduces from $2.9817\times 10^{-4}$ to $9.1456\times 10^{-7}$.
The associated maximum error $\mathcal{E}_\infty$ decreases from $0.0564$ to $7.3064\times 10^{-4}$ as expected.
Since low dimensional cases are more interesting in practice, in the rest of this example, we will focus on improving the numerical performance of $40$-dimensional SP-ROMs by correcting the energy approximation.

\subsubsection{SP-ROMs with corrected energy}
\paragraph{Approach I. SP-ROMs with enriched POD basis (SP-ROM-1)}
In this approach, one new basis function $\bpsi$ generated from the residual of the initial data is added to the basis set $\bPhi$, i.e., $\widetilde{\bPhi}= [\bPhi, \bpsi]$.

The discrete structure-preserving ROM with enriched POD basis (SP-ROM-1) reads
\begin{equation}
\frac{\ba^{k+1}-\ba^k}{\Delta t} =
\bA_r\widetilde{\bPhi}^\top\left[\frac{\alpha}{6}\left( (\widetilde{\bPhi}\ba^k)^2
+(\widetilde{\bPhi}\ba^k)(\widetilde{\bPhi}\ba^{k+1})
+(\widetilde{\bPhi}\ba^{k+1})^2 \right)
+ \rho \widetilde{\bPhi} \frac{\ba^k+\ba^{k+1}}{2}
+ \nu\bB\widetilde{\bPhi} \frac{\ba^k+\ba^{k+1}}{2}\right]
\label{eq: kdv-sp-rom-1}
\end{equation}
with $\ba^0 = \widetilde{\bPhi}^\top \bu^0$.

Since the residual information of initial data is included in the enriched basis, it is expected that the energy at initial time can be exactly captured, thus the structure-preserving ROM can preserve the exact energy.
The $41$-dimensional SP-ROM-1 \eqref{eq: kdv-sp-rom-1} is simulated, whose energy approximation error is shown in Figure \ref{Fig: kdv_sprom1_r40} (left).
It is observed that energy error $\mathcal{H}_r(t)-\mathcal{H}(t)\sim\mathcal{O}(10^{-12})$.
Meanwhile, the maximum error $\mathcal{E}_\infty= 0.050168$.
\begin{figure}[htb]
\centering
\begin{minipage}[ht]{0.41\linewidth}
\includegraphics[width=1\textwidth]{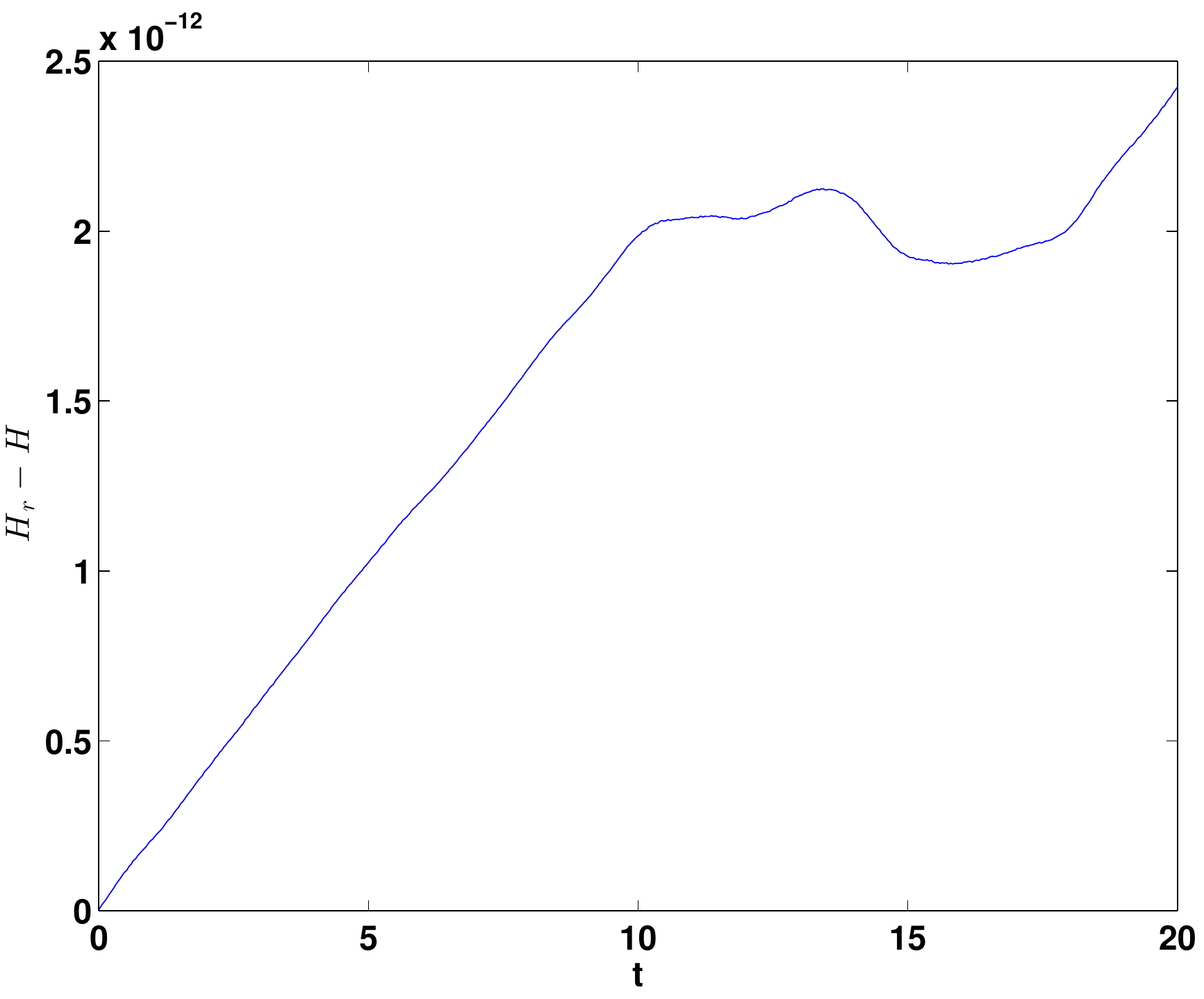}
\end{minipage}
\hspace{1cm}
\begin{minipage}[ht]{0.41\linewidth}
\includegraphics[width=1\textwidth]{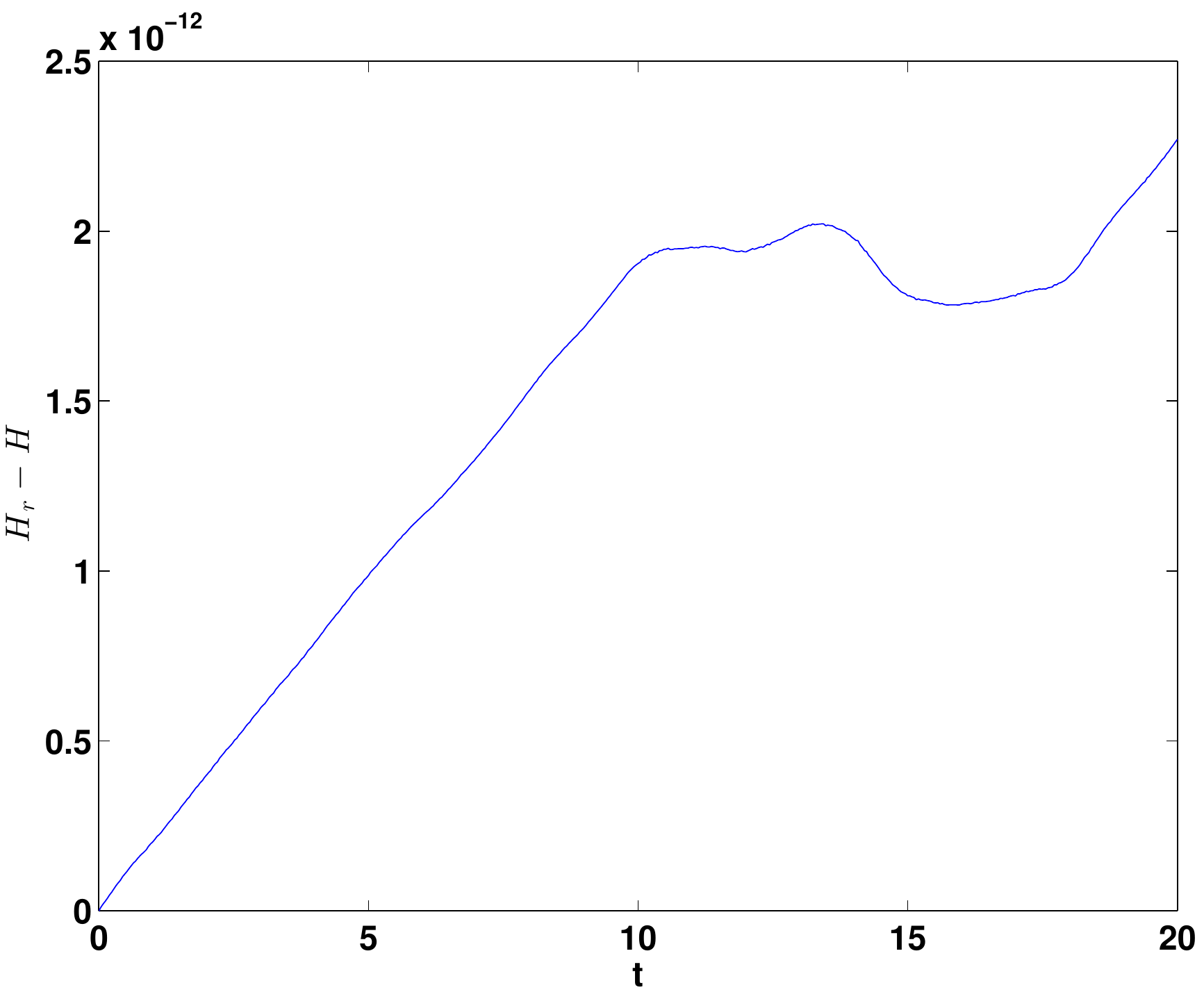}
\end{minipage}
\caption{Time evolution of energy error in the 41-dimensional SP-ROM-1 (left) and 40-dimensional SP-ROM-2 (right) when $\mu=0$. Note that the magnitude of $\mathcal{H}_r(t)-\mathcal{H}(t)$ is $\mathcal{O}(10^{-12})$, thus the energy is accurately captured in both models.
}\label{Fig: kdv_sprom1_r40}
\end{figure}

\paragraph{Approach II. SP-ROMs with POD basis from shifted snapshots}
In this approach, the POD basis is extracted from the shifted snapshots \eqref{eq: snap_shift}.
The corresponding structure-preserving ROM with the new POD basis (SP-ROM-2) is to find $\ba^k$ satisfying
\begin{eqnarray}
\frac{\ba^{k+1}-\ba^k}{\Delta t} &=&
\bA_r\bPhi^\top\Big[\frac{\alpha}{6}\left( (\bu_0+\bPhi\ba^k)^2
+(\bu_0+\bPhi\ba^k)(\bu_0+\bPhi\ba^{k+1})
+(\bu_0+\bPhi\ba^{k+1})^2 \right) \nonumber \\
&+&  \rho (\bu_0+\bPhi \frac{\ba^k+\ba^{k+1}}{2})
+ \nu\bB(\bu_0+\bPhi \frac{\ba^k+\ba^{k+1}}{2})\Big]
\label{eq: kdv-sp-rom-2}
\end{eqnarray}
with $\ba^0 = {\bf 0}$.

Consider the $r= 40$ case, the maximum error of the SP-ROM-2 simulation is $\mathcal{E}_\infty = 0.036574$.
Time evolution of the energy error is plotted in Figure \ref{Fig: kdv_sprom1_r40} (right), which indicates no any visible discrepancy of the energy between the ROM and FOM since $\mathcal{H}_r(t)-\mathcal{H}(t)\sim \mathcal{O}(10^{-12})$.


\subsection{Summary of numerical experiments}
We summarize the test cases in Tables \ref{tab: ex1}-\ref{tab: ex2}.
\begin{table}[htp]
\begin{center}
\caption{Linear wave equations (Section \ref{sec: we}): comparison of ROMs (r=5)}
\label{tab: ex1}
\begin{tabular}{| c | c | c | c | c | c | c |} \hline
  {}                              & FOM  &   G-ROM      & SP-ROM-0 ($\mu=0$)  & SP-ROM-1 ($\mu=0$)      & SP-ROM-2  ($\mu=0$) \\ \hline
  $\mathcal{E}_\infty$ &  --      & 0.4591         & 0.2606         & 0.4138              & 0.1526      \\
  $H_r$    			  & 0.075 & nonconst.    & 0.06788         & 0.075                & 0.075 \\     \hline
\end{tabular}
\end{center}
\end{table}
\begin{table}[htp]
\begin{center}
\caption{Korteweg-de Vries equations (Section \ref{sec: kdv}): comparison of ROMs (r=40)}
\label{tab: ex2}
\begin{tabular}{| c | c | c | c | c | c | c |} \hline
  {}                              & FOM  &   G-ROM      & SP-ROM-0  ($\mu=0$)  & SP-ROM-1 ($\mu=0$)    & SP-ROM-2 ($\mu=0$) \\ \hline
  $\mathcal{E}_\infty$ &  --       & 0.02964         & 0.0564         & 0.05017              & 0.03657      \\
  $H_r$    			  & -1.1317 & nonconst.    & -1.1314         & -1.1317                & -1.1317 \\     \hline
\end{tabular}
\end{center}
\end{table}

Based on the preceding two test experiments, we draw the following conclusions:
(i) the structure-preserving ROMs (SP-ROM-0, SP-ROM-1, and SP-ROM-2) are able to keep the energy a constant;
(ii) choosing an optimal weight $\mu$ in the snapshot set helps improve the accuracy of reduced-order approximations, but a more straightforward choice $\mu= 0$ can yields a good approximation without paying the price of tuning the free parameter;
(iii) the structure-preserving ROMs with corrected energy (SP-ROM-1 and SP-ROM-2) are able to achieve  accurate Hamiltonian function; while SP-ROM-2 has the better accuracy in the state variable approximation than SP-ROM-1;
(iv) overall, SP-ROM-2 obtains more accurate state variable solution and energy approximation, which outperforms the other ROMs discussed in this paper.

\section{Conclusions}
\noindent \indent One of the most important features of Hamiltonian PDE systems is possessing some invariant Hamiltonian functions, which represent important physical quantities such as the total energy of the system.
The standard Galerkin projection-based POD-ROM is not able to inherit this property in its discretized system.
In this paper, we develop new structure-preserving POD-ROMs for Hamiltonian PDE systems, which use the same Galerkin projection strategy, but keep the Hamiltonian constant by developing a new coefficient matrix for the reduced-order dynamical system.
With the use of the POD basis from shifted snapshots, the structure-preserving ROM (SP-ROM-2) is able to produce an exact Hamiltonian and an accurate approximation of the state variables.
Numerical experiments demonstrate the effectiveness of the proposed method.
This approach can be easily extended to other systems such as port-Hamiltonian systems and dissipative gradient systems, which will be a explored in our future research.


\section{Acknowledgements}
Y. Gong's work is partially supported by China Postdoctoral Science Foundation through Grants 2016M591054.
Q. Wang's research is partially supported by NSF grants DMS-1200487 and DMS-1517347, as well as an NSFC grant 11571032. 
Z. Wang is grateful for the supports from the National Science Foundation through Grants DMS-1522672, and the office of the Vice President for Research at University of South Carolina through the ASPIRE-I program. 

\bibliographystyle{unsrt} 
\bibliography{bib_structure,bib_thesis}
\end{document}